\documentclass[svgnames]{siamart1116}

\usepackage{lipsum,xspace}
\usepackage{amsfonts}
\usepackage{graphicx}
\usepackage{subcaption}
\usepackage{bbm}
\usepackage{multicol} \usepackage{mathtools}
\usepackage{threeparttable}
\usepackage{rotating}
\usepackage{multirow}
\usepackage{amssymb}
\usepackage{relsize}
\usepackage{algorithm}
\usepackage{algorithmic}

\usepackage{color, colortbl}
\definecolor{Gray}{gray}{0.9}

\newcommand{\subscr}[2]{#1_{\textup{#2}}}
 \newcommand{\setdef}[2]{\{#1
  \; | \; #2\}}

\newcommand{\map}[3]{#1: #2 \rightarrow #3}

\newcommand{\mcA}{\mathcal{A}}

\newcommand{\Bt}{B^\top}

\newcommand{\mcP}{\mathcal{P}}

\newcommand{\norm}[2]{\left\|#1\right\|_{\mathrm{#2}}}

\newcommand{\Biggnorm}[2]{\Bigg\|#1\Bigg\|_{\mathrm{#2}}}

\newcommand{\FCmap}{\mathcal{F}_{\complex}}
\newcommand{\Fmap}{\mathcal{F}}

\newcommand\oprocendsymbol{\hbox{$\triangle$}}
\newcommand\oprocend{\relax\ifmmode\else\unskip\hfill\fi\oprocendsymbol}

\usepackage{enumitem} \renewcommand{\theenumi}{(\roman{enumi})}

\DeclareSymbolFont{bbold}{U}{bbold}{m}{n}
\DeclareSymbolFontAlphabet{\mathbbold}{bbold}
\newcommand{\vect}[1]{\mathbbold{#1}}
\newcommand{\vectorones}[1][]{\vect{1}_{#1}}
\newcommand{\vectorzeros}[1][]{\vect{0}_{#1}}

\newcommand{\scirc}{\raise1pt\hbox{$\,\scriptstyle\circ\,$}}
\newcommand{\real}{\mathbb{R}}
\newcommand{\complex}{\mathbb{C}}

\newcommand{\torus}{\mathbb{T}} 
\newcommand{\Z}{\mathbb{Z}}

\newtheorem{remark}{Remark}

\ifpdf
\DeclareGraphicsExtensions{.pdf,.eps,.png,.jpg} \else
\DeclareGraphicsExtensions{.eps} \fi

 \newcommand{\imagunit}{\mathrm{i}}

\newcommand{\prjcut}{\mathcal{P}_{\mathrm{cut}}}
\newcommand{\prjcyc}{\mathcal{P}_{\mathrm{cyc}}}

\DeclareMathOperator{\diag}{diag}

\DeclareMathOperator{\Ker}{\mathrm{Ker}}
\DeclareMathOperator{\Img}{\mathrm{Img}}

\newcommand{\ErdosRenyi}{Erd\H{o}s\textendash{}R\'{e}nyi\xspace}

\newcommand{\TheTitle}{Synchronization Tests for the Kuramoto Model
  via Power Series} \renewcommand{\TheTitle}{Synchronization of
  Kuramoto Oscillators: \newline Inverse Taylor
  Expansions}
\newcommand{\TheAuthors}{Saber Jafarpour and Elizabeth Y. Huang and Francesco Bullo}

\headers{Synchronization of Kuramoto Oscillators}{\TheAuthors}

\title{{\TheTitle}\thanks{This work was supported in part by the
    U.S. Department of Energy (DOE) Solar Energy Technologies Office
    under Contract No. DE-EE0000-1583.}}

\author{Saber Jafarpour\thanks{Center for Control, Dynamical Systems, and
    Computation, University of California, Santa Barbara
    (\email{saber.jafarpour@engineering.ucsb.edu}).} \and Elizabeth
  Y. Huang\thanks{Center for Control, Dynamical Systems, and Computation,
    University of California, Santa Barbara
    (\email{eyhuang@engineering.ucsb.edu}).}  \and Francesco
  Bullo\thanks{Department of Mechanical Engineering and Center for Control,
    Dynamical Systems, and Computation, University of California, Santa
    Barbara (\email{bullo@engineering.ucsb.edu})} }

\ifpdf \hypersetup{ pdftitle={\TheTitle}, pdfauthor={\TheAuthors} }
\fi

\begin{document}

\maketitle

\begin{abstract} 
  Synchronization in networks of coupled oscillators is a widely
  studied topic with extensive scientific and engineering
  applications.  In this paper, we study the frequency synchronization
  problem for networks of Kuramoto oscillators with arbitrary topology
  and heterogeneous edge weights.  We propose a novel equivalent
  transcription for the equilibrium synchronization equation. Using
  this transcription, we develop a power series expansion to compute
  the synchronized solution of the Kuramoto model as well as a
  sufficient condition for the strong convergence of this series
  expansion.  Truncating the power series provides (i) an efficient
  approximation scheme for computing the synchronized solution, and
  (ii) a simple-to-check, statistically-correct hierarchy of
  increasingly accurate synchronization tests.  This hierarchy of
  tests provides a theoretical foundation for and generalizes the
  best-known approximate synchronization test in the literature.  Our
  numerical experiments illustrate the accuracy and the computational
  efficiency of the truncated series approximation compared to
  existing iterative methods and existing synchronization tests.
\end{abstract}




\section{Introduction}

Collective synchronization is an interesting behaviour which lies at
the heart of various natural phenomena. The celebrated Kuramoto
model~\cite{YK:75} is one of the simplest models for studying
synchronization in a network of coupled oscillators. Kuramoto model
has been successfully used to model the synchronization behaviour of a
wide range of physical, chemical, and biological
system~\cite{JAA-LLB-CJPV-FR-RS:05}. Examples include the power
grids \cite{DJH-GC:06,FD-FB:09z}, automated vehicle coordination
\cite{DJK-PL-KAM-TJ:08, RS-DP-NEL:07}, pacemakers in
heart~\cite{YW-FJD:13}, clock synchronization \cite{OS-US-YBN-SS:08},
and neural networks \cite{GBE-NK:91}; see
also~\cite[Chapter~13]{FB:18} for additional examples.  One of the
most interesting types of synchronization is frequency
synchronization, where all oscillators reach the same rotational
frequency with possibly different phases. It is well-known that the
Kuramoto model can exhibit a transition from incoherence to frequency
synchronization. For many applications, such as power networks, it is
important to have an accurate estimate of this transition to
synchronization. This is essential, particularly as the grid is being
pushed closer to its maximum capacity due to increases in the load
demand and penetration of renewable energy units. Finding sharp
conditions to determine when this transition happens continues to be a
challenging as well as a critical problem.

\subsection*{Literature Review} 

The problem of finding conditions for existence of a stable synchronized
solution for the Kuramoto model of coupled oscillators has been studied
extensively in the literature.  For complete graphs with homogeneous
weights, the order parameter is used to implicitly determine the exact
critical coupling needed for a synchronized
solution~\cite{DA-JAR:04,REM-SHS:05,MV-OM:08}. For acyclic graphs with
heterogeneous weights, a necessary and sufficient condition is developed
for synchronization of the heterogeneous Kuramoto
model~\cite{FD-MC-FB:11v-pnas}. In addition, Lyapunov analysis applied to
the complete graph is used in~\cite{NC-MWS:08} to give a sufficient
condition and in~\cite{FD-FB:10w} to give an explicit necessary and
sufficient condition for existence of a synchronized solution.  However for
general topology graphs, such a complete characterization of frequency
synchronization does not exist.  For general graphs with heterogeneous
weights, several necessary conditions and sufficient conditions for
existence of stable synchronized solution have been reported in the
literature. \cite{CJT-OJMS:72a} requires sufficiently large nodal
degrees relative to the natural frequencies, \cite{NA-SG:13} uses the
cutset in the graph, and \cite{AJ-NM-MB:04} states that the
algebraic connectivity must be sufficiently large compared to the
difference in natural frequencies of connected oscillators. Recently, a
novel cutset projection operator has been introduced to rigorously prove a
simple-to-check, sufficient condition for synchronization of Kuramoto
model~\cite{SJ-FB:16h}. Using numerous simulations, it is shown that this
new sufficient condition scales better to large networks~\cite{SJ-FB:16h}.

Despite these deep results in the literature, the existing synchronization
conditions usually provide conservative estimates for the synchronization
threshold. In an effort to come closer to finding the exact synchronization
threshold, \cite{FD-MC-FB:11v-pnas}~and~\cite{EYH-SJ-FB:18d} introduce a
statistically accurate approximate test for synchronization that depends on
the network parameters and topology derived from the linearized Kuramoto
map and the converging power series expansion of the phase angles of the
Kuramoto oscillators, respectively.

If existence of a frequency synchronized solution can be guaranteed,
then the next step is computing the synchronized solutions.  A common
method to approximate the solution is to linearize the equations. This
will result in studying the equations of the form
$\omega=L{\bf{\theta}}$, where $L$ is the Laplacian matrix of the
network~\cite{RK-SS:16,DAS-SHT:14,BS-JJ-OA:09}. The angles $\theta$
can be approximately solved very efficiently, even for extremely
large, sparse graphs~\cite[Theorem 3.1]{NKV:13}. However, when phase
differences of the oscillators are large, this linear approximation is
not very accurate.

In order to compute the synchronization manifold of the nonlinear
Kuramoto equations, one can employ iterative numerical algorithms such
as Newton\textendash{}Raphson or
Gauss\textendash{}Seidel~\cite{BS-OA:74,WFT-CEH:67,AFG-GWS:57}. Unfortunately,
these algorithms do not guarantee convergence to the synchronized
solutions and failure of these algorithms could be due to numerical
instability, an initialization issue, or non-existence of the
solution. Another approach is to use numerical polynomial homotopy
continuation (NPHC). It is guaranteed that NPHC will find all stable
and unstable manifolds of the Kuramoto model, but this method is not
computationally tractable for large networks; \cite{DM-NSD-FD:15} uses
NPHC to study the homogeneous Kuramoto model for particular graph
topologies with up to $18$ nodes.  Finally, \cite{CW-NR-CG-MSB:15}
gives an approximate analytical solution for stable synchronization
manifolds using the order parameter. However, this approximation
scheme is only applicable to the uniform-weight Kuramoto model with
all-to-all connections.

A wide range of methods for finding synchronized solutions of the
Kuramoto model stem from the power network literature, where different
techniques are used to find the solutions of the AC power flow
equations. Here, we only review two of these approaches. The first
approach is called Holomorphic Embedding Load-Flow Method (HELM) and
has been proposed to find all the solutions of power flow
equations~\cite{AT:12}. While HELM is based on advanced results and
concepts from complex analysis, its numerical implementation is
recursive and straightforward~\cite{AT:12,SR-YF-DJT-MKS:16}.  However,
HELM is reported to be much slower than the Newton\textendash{}Raphson
methods~\cite{SR-YF-DJT-MKS:16}.  The second approach is the
optimization approach, whereas an optimal power flow problem (OPF) is
used to solve for the AC power flow equations. The OPF problems have
been studied extensively in the power network
literature, e.g., see~\cite{JAM-RA-MEE:99-1,
  JAM-RA-MEE:99-2,DKM-FD-HS-SHL-SC-RB-JL:17}. Thus, one can use the
numerical algorithms for the optimization problem to find the
synchronized solution of the Kuramoto model. Unfortunately, due to the
non-convex nature of the OPFs, these algorithms usually result in an
approximation of the synchronized solution.

\subsection*{Contribution} 
The contributions of this paper are both theoretical and
computational. From a theoretical viewpoint, first, we review
important properties of the Kuramoto model of coupled oscillators and,
as a minor contribution, we provide a rigorous proof for the following
well-known folk theorem: frequency synchronization is equivalent with
the existence of a stable synchronization manifold
(see~\cite{REM-SHS:05} and~\cite{MV-OM:08} for statement of this
result without proof).  Second, by introducing the notions of edge
vectors and flow vectors in graphs, we propose four equivalent
transcriptions for the synchronization manifold of the Kuramoto model:
node, flow, constrained edge, and unconstrained edge balance
equations. While the first three formulations have already been
studied in the literature (see~\cite{FD-MC-FB:11v-pnas}
and~\cite{SJ-FB:16h}), the unconstrained edge balance equations
provide a novel important characterization of the synchronization
manifold. Our main technical results are (1) a sufficient condition
for existence of a unique solution for unconstrained edge balance
equations and (2) a recursive expression for each term of the Taylor
series expansion for this solution of the unconstrained edge balance
equations. Additionally, we prove that, if our simple-to-check
sufficient condition is satisfied, then the Taylor series expansion
for the solution converges strongly.  We also provide an algorithm to
symbolically compute all terms of the expansion.  Third and final,
using the one-to-one correspondence between solutions of the
unconstrained edge balance equations and synchronized solutions of the
Kuramoto model, we propose a power series expansion for the
synchronized solutions of the Kuramoto model and an estimate on the
region of convergence of the power series.

From a computational viewpoint, first, we propose a method to approximate
the synchronization manifold of the Kuramoto model using the truncated
power series. We present several numerical experiments using IEEE test
cases and random graphs to illustrate (1) the accuracy of the truncated
series and (2) the computational efficiency of the new methods for
computing the synchronization manifold.  We show that the seventh order approximate
method has low absolute error when applied to IEEE test cases with weakly
coupled oscillators.  The truncated series, up to the seventh order, have
comparable computational efficiency to Newton\textendash{}Raphson when solving for
solutions with static graph topology and multiple natural frequencies, or
power injections.
Second, based on our novel power series approach, we
  propose a hierarchy of approximate tests for synchronization of the
  Kuramoto model; our approach provides a theoretical basis for and
  generalize the state-of-the-art approximate synchronization test in
  the literature~\cite{FD-MC-FB:11v-pnas}. With numerical analysis,
we verify the accuracy of our family of approximate tests for several random graphs and numerous IEEE test cases.  In each of these cases,
we show that our new approximate tests are a significant improvement
compared to the best-known approximate condition given
in~\cite{FD-MC-FB:11v-pnas}.


Finally, we compare this paper with our preliminary conference
article~\cite{EYH-SJ-FB:18d}. In short, this paper presents a substantially
more complete and comprehensive treatment of the power series approach to
synchronization of Kuramoto oscillators. Specifically,
while~\cite{EYH-SJ-FB:18d} presents a power series expansion for nodal phase
angles, this paper develops a novel power series expansion for the flows in
the network. Using the Banach Fixed-Point Theorem, we provide an estimate
on the domain of convergence of the power series which is substantially
larger than the estimates given in~\cite{EYH-SJ-FB:18d}. Moreover, our
numerical analysis shows that the hierarchy of approximate synchronization
tests obtained by truncating this power series is more accurate than the
estimate tests proposed in~\cite{EYH-SJ-FB:18d}.

\subsection*{Paper organization} 
In Section~\ref{sec:prelim}, we give preliminaries and notation used in the
paper. In Section~\ref{sec:kuramoto}-\ref{sec:auxiliary} we review the
Kuramoto model, frequency synchronization, and give several equivalent
formulations of the algebraic Kuramoto equation.
Sections~\ref{sec:cut-flow-balance-solution} and~\ref{sec:syncOfKuramoto}
contain the paper's main theoretical results and a family of approximate synchronization tests.
Finally, Section~\ref{sec:numerical} contains numerical experiments analyzing the approximate synchronization tests and efficiency of computation methods for the synchronization manifold.

\section{Preliminaries and notation}\label{sec:prelim}

\subsection*{Vectors and functions}
Let $\Z_{\ge 0}$, $\real^n$, and $\complex^n$ denote the set of
non-negative integers, the $n$-dimensional real Euclidean space, and the
$n$-dimensional complex Euclidean space, respectively. For $n\in
\Z_{\ge 0}$, let $n!! = \prod_{k=0}^{\lceil\frac{n}{2}\rceil-1}
(n-2k)$ denote the double factorial.  For $r>0$ and $\mathbf{x}\in
\real^n$, the real polydisk with center $\mathbf{x}$ and radius $r$ is
\begin{align*}
\mathrm{D}_n(\mathbf{x},r) =\setdef{\mathbf{y}\in
  \real^n}{\|\mathbf{x}-\mathbf{y}\|_{\infty}\le r}.
\end{align*} 
Similarly, for $r>0$ and $\mathbf{z}\in \complex^n$, the complex polydisk
with center $\mathbf{z}$ and radius $r$ is 
\begin{align*}
\mathrm{D}^{\complex}_n(\mathbf{z},r) =\setdef{\mathbf{w}\in
  \complex^n}{\|\mathbf{z}-\mathbf{w}\|_{\infty}\le r}.
\end{align*} 
Let $ \vectorones[n] $ and $ \vectorzeros[n] $ be $n$-dimensional column vectors of ones and zeros respectively.
For $ \mathbf{x} = (x_1,\dots,x_n)^{\top} \in\complex^{n} $, let $
\sin(\mathbf{x}) = (\sin(x_1),\dots,\sin(x_n))^{\top} $ and
$\diag(\mathbf{x})$ be the $n\times n$ diagonal matrix with
$\left(\diag(\mathbf{x})\right)_{ii} = x_i$, for every $i\in \{1,\ldots,n\}$. For $\mathbf{x} = (x_1,\dots,x_n)^{\top} \in\complex^{n} $ with
$\|\mathbf{x}\|_{\infty}\le 1$, let $\arcsin(\mathbf{x}) =
(\arcsin(x_1),\dots,\arcsin(x_n))^{\top} $, where 
\begin{align*}
\arcsin(r) = \sum_{i=0}^{\infty} \frac{(2i-1)!!}{(2i)!! (2i+1)} r^{2i+1}.
\end{align*}
For every $n\in \mathbb{N}$, we denote the $n$-torus by $\mathbb{T}^n$. For
every $s\in [0,2\pi)$, the clockwise rotation of $\theta\in \mathbb{T}^n$
  by the angle $s$ is the function
  $\map{\mathrm{rot}_s}{\mathbb{T}^n}{\mathbb{T}^n}$ defined by
\begin{equation*}
\mathrm{rot}_s(\theta)=(\theta_1+s,\ldots,\theta_n+s)^{\top},\qquad\text{for }
\theta\in \mathbb{T}^n.
\end{equation*}
Using the rotation function, one defines an equivalence relation
$\sim$ on the $n$-torus $\mathbb{T}^n$ as follows: For every two points $\theta,\eta\in
\mathbb{T}^n$, we say $\theta\sim \eta$ if there exists $s\in
[0,2\pi)$ such that $\theta = \mathrm{rot}_s(\eta)$. For every $\theta\in \mathbb{T}^n$, the equivalence
class of $\theta$ is denoted by $[\theta]=\left\{\mathrm{rot}_s(\theta)\mid s\in
  [0,2\pi)\right\}$. The quotient space of $\mathbb{T}^n$ under the
equivalence relation $\sim$ is denoted by $[\mathbb{T}^n]$.

\subsection*{Algebraic graph theory}
Let $G$ be a weighted undirected connected graph with the node set
$\mathcal{N}=\{1,\ldots,n\}$ and the edge set $\mathcal{E}\subseteq
\mathcal{N}\times \mathcal{N}$ with $m$ elements. We assume that $G$ has no
self-loops and the weights of the edges are described by the nonnegative,
symmetric adjacency matrix $ A\in\real^{n\times n} $. The Laplacian matrix
of the graph $G$ is $ L=\diag(A\vectorones[n])-A \in\real^{n\times n}$.
Define the diagonal edge weight matrix by
$\mcA=\diag(a_{ij\in\mathcal{E}})\in\real^{m\times m}$. It is known that the
Laplacian is $ L = B\mcA\Bt $. Since $ L $ is singular, we use the
Moore\textendash{}Penrose pseudoinverse $ L^{\dagger} $ which has the
following properties: $ LL^{\dagger}L=L $, $
L^{\dagger}LL^{\dagger}=L^{\dagger} $, $ L^{\dagger}L=(L^{\dagger}L)^{\top}
$, and $ LL^{\dagger}=(LL^{\dagger})^{\top} $. In addition, for a connected
graph $ L^{\dagger}L=LL^{\dagger}=I_n -
\frac{1}{n}\vectorones[n]^\top\vectorones[n] $. The \emph{weighted cutset
  projection matrix} $\prjcut$ is the oblique projection onto
$\Img(B^{\top})$ parallel to $\Ker(B\mcA)$ given by
\begin{align*}
  \prjcut=B^{\top}L^{\dagger}B\mcA.
\end{align*}
The weighted cutset projection matrix $\prjcut$ is idempotent, and $0$ and
$1$ are its eigenvalues with algebraic (and geometric) multiplicity $m-n+1$ and
$n-1$, respectively. Additional properties of $ \prjcut$ are
in~\cite[Theorem 5]{SJ-FB:16h}. Similarly, the \emph{weighted cycle
  projection matrix} $\prjcyc$ is the oblique projection onto
$\Ker(B\mcA)$ parallel to $\Img(B^{\top})$ given by
\begin{align*}
  \prjcyc=I_m - B^{\top}L^{\dagger}B\mcA.
\end{align*}

\subsection*{Analytic functions and power series}

A multi-index $\nu$ is a member of
$\left(\Z_{>0}\right)^n$. For every $x\in \mathbb{C}^n$, we
define $x^{\nu}=x_1^{\nu_1}x_2^{\nu_2}\ldots x_n^{\nu_n}$.  For
$x_0\in \complex^n$, the formal expression
\begin{equation}\label{eq:5}
\sum_{\nu\in \left(\Z_{>0}\right)^n} a_{\nu}(x-x_0)^{\nu},
\end{equation}
where $a_{\nu}\in \mathbb{C}$, for every $\nu\in
\left(\Z_{>0}\right)^n$ is called a formal power series around
point $x_0$. 
The power series $\sum_{\nu\in \left(\Z_{>0}\right)^n}
a_{\nu}(x-x_0)^{\nu}$ \emph{converges strongly} at point $x$ if all
rearrangement of the terms of the series
$\sum_{\nu}a_{\nu}(x-x_0)^{\nu}$ converges.
For every $x_0\in \mathbb{C}^n$, the \emph{domain of convergence of}
\eqref{eq:5} around $x_0$ is defined as the set $\mathcal{C}_{x_0}$ of
all points $x\in \mathbb{C}^n$ such that the power series
$\sum_{\nu}a_{\nu}(x-x_0)^{\nu}$ converges strongly at point $x$.
While for $n=1$, one can show that the domain of convergence is an
open interval around $x_0$, for $n>1$ the domain of convergence of a
power series is not necessarily an open poly-disk around $x_0$. An open set $\Omega\subset \complex^n$ is a \emph{Reinhardt domain} if, for
every $\begin{pmatrix}z_1,\ldots,z_n\end{pmatrix}^{\top}\in \Omega$
and every $\begin{pmatrix}\theta_1,\ldots,\theta_n\end{pmatrix}^{\top}\in
\mathbb{T}^n$, we have $\begin{pmatrix}e^{i\theta_1}z_1,\ldots,e^{i\theta_n}z_n\end{pmatrix}^{\top}\in
\Omega$. The Reinhardt domains can be considered as the generalization of the
disks on the complex plane to higher dimensions.

\section{The heterogeneous Kuramoto model}\label{sec:kuramoto}

The Kuramoto model is a system of $n$ oscillators, where each
oscillator has a natural frequency $\omega_i\in \real$ and its state
is represented by a phase angle $\theta_i\in \mathbb{S}^1$. The
interconnection of these oscillators are described using a weighted
undirected connected graph $G$, with nodes
$\mathcal{N}=\{1,\ldots,n\}$, edges $\mathcal{E}\subseteq
\mathcal{N}\times \mathcal{N}$, and positive weights $a_{ij}=a_{ji}>0$. The dynamics for the heterogeneous
Kuramoto model is given by:
\begin{equation}\label{eq:2}
\dot{\theta}_i=\omega_i-\sum_{j=1}^{n}a_{ij}
\sin(\theta_i-\theta_j),\qquad\text{for } i\in\{1,\ldots,n\}.
\end{equation}
In matrix language, one can write this differential equations as:
\begin{equation}\label{eq:kuramoto_model}
\dot{\theta}=\omega-B\mcA\sin(B^{\top}\theta),
\end{equation}
where $\theta=(\theta_1,\theta_2,\ldots,\theta_n)^{\top}\in \mathbb{T}^n$
is the phase vector, $\omega=(\omega_1,\omega_2,\ldots,\omega_n)^{\top}\in
\real^n$ is the natural frequency vector, and $B$ is the incidence matrix
for the graph $G$. One can show that if
$\map{\theta}{\mathbb{R}_{\ge0}}{\mathbb{T}^n}$ is a solution for the
Kuramoto model~\eqref{eq:kuramoto_model} then, for every $s\in[0,2\pi)$,
  the curve $\map{\mathrm{rot}_s(\theta)}{\real_{\ge 0}}{\mathbb{T}^n}$ is
  also a solution of~\eqref{eq:kuramoto_model}. Therefore, for the rest of
  this paper, we consider the state space of the Kuramoto
  model~\eqref{eq:kuramoto_model} to be $[\mathbb{T}^n]$.

\begin{definition}[\textbf{Frequency synchronization}]
A solution $\map{\theta}{\mathbb{R}_{\ge 0}}{[\mathbb{T}^n]}$ of the
coupled oscillator model \eqref{eq:kuramoto_model} achieves \emph{frequency
  synchronization} if there exists a frequency $\omega_{\mathrm{syn}}\in
\real$ such that
\begin{equation*}
\lim_{t\to\infty} \dot{\theta}(t)=\omega_{\mathrm{syn}}\vect{1}_n.
\end{equation*}
\end{definition}
By summing all the equations in~\eqref{eq:2}, one
can show that if a solution of \eqref{eq:kuramoto_model} achieves
frequency synchronization then $\subscr{\omega}{syn} =
\tfrac{1}{n}\sum_{i=1}^{n} \omega_i$. Therefore, without loss of generality, we can assume that in the Kuramoto
model~\eqref{eq:kuramoto_model}, we have $\omega\in
\vect{1}^{\perp}_n$ and $\omega_{\mathrm{syn}}=0$.
\begin{definition}[\textbf{Synchronization manifold}]
Let $\theta^*$ be a solution of the algebraic equation
\begin{equation}\label{eq:synchronization_manifold}
\omega=B\mcA\sin(B^{\top}\theta^*).
\end{equation}
Then $[\theta^*]$ is called a \emph{synchronization manifold} for the
Kuramoto model \eqref{eq:kuramoto_model}.
\end{definition}
The following theorem reduces the problem of local frequency
synchronization in the Kuramoto model~\eqref{eq:kuramoto_model} to the
existence of a solution for the algebraic
equations~\eqref{eq:synchronization_manifold}. 
\begin{theorem}[\textbf{Characterization of frequency synchronization}]\label{thm:9}
For the heterogeneous Kuramoto model~\eqref{eq:kuramoto_model}
on graph $G$, the following statements are equivalent:
\begin{enumerate}
\item\label{p1:frequency} there exists an open set $U\in [\mathbb{T}^n]$
  such that every solution of the Kuramoto
  model~\eqref{eq:kuramoto_model} starting in set $U$ achieves frequency synchronization;
\item\label{p2:equilibrium} there exists a locally asymptotically
  stable synchronization manifold $[\theta^*]$ for~\eqref{eq:kuramoto_model}.
\end{enumerate} 
Additionally, if any of equivalent
conditions~\ref{p1:frequency} or~\ref{p2:equilibrium} holds, then, for
every $\theta(0)\in U$, we have $\lim_{t\to\infty} [\theta(t)] = [\theta^*]$.
\end{theorem}
\begin{proof}
  Regarding $\ref{p1:frequency}\implies\ref{p2:equilibrium}$, if the
  solution achieves frequency synchronization, then
  $\lim_{t\to\infty}\dot{\theta}_i(t) = 0 = \lim_{t\to\infty}
  \Big(\omega_i-\sum_{j=1}^{n}a_{ij}\sin(\theta_i(t)-\theta_j(t))\Big)$ for
  all $i=\{1,...,n\}$.  Consider a sequence of natural numbers
  $k\in\mathbb{N}$ and the corresponding sequence $\theta(k)$ in
  $\mathbb{T}^n$. Since $\mathbb{T}^n$ is a compact metric space, it is
  sequentially compact~\cite[Theorem 28.2]{JM:00}. This means that there is
  a subsequence $\hat{k}$ such that $\theta_i(\hat{k})$ is convergent.
  Then $\lim_{\hat{k}\to\infty}\theta_i(\hat{k})$ exists and
  $\vectorzeros[n]=\omega-B\mcA\sin(B^{\top}\lim_{\hat{k}\to\infty}\theta(\hat{k}))$.
  Therefore $ \lim_{\hat{k}\to\infty}\theta(\hat{k}) $ is a synchronization
  manifold because it is a solution for
  equation~\eqref{eq:synchronization_manifold} and is locally
  asymptotically stable since all solutions starting in $U$ reach
  $\lim_{\hat{k}\to\infty}\theta(\hat{k})$.

Regarding $\ref{p2:equilibrium}\implies\ref{p1:frequency}$, by the
definition of local asymptotic stability, there exists some $ \delta>0 $
such that the open set $ U $ is defined to be
$U=\setdef{\theta(0)\in[\mathbb{T}^n]}{\|\theta(0)-\theta^*\|\leq\delta}$
where $[\theta^*]$ is the synchronization manifold.  Then for solutions
starting in $U$, $\lim_{t\to\infty}\dot{\theta}(t) =
\omega-B\mcA\sin(B^{\top}\theta^*) = 0$ so $[\theta^*]$ is also a frequency
synchronized solution for \eqref{eq:kuramoto_model}.
	
The last statement follows from the proofs of
$\ref{p1:frequency}\implies\ref{p2:equilibrium}$ and
$\ref{p2:equilibrium}\implies\ref{p1:frequency}$.
\end{proof}

In many application, such as power networks, not only is it important to
study the frequency synchronization of the Kuramoto oscillators but also it
is essential to bound the position of the synchronization manifold
$[\theta^*]$ due to some security constraints for the grid. An
important class of security constraints are thermal constraints which are
usually expressed as bounds on the geodesic distances
$|\theta^*_i-\theta^*_j|$, for $i,j\in \{1,\ldots,n\}$. The geodesic
distance $|\theta^*_i-\theta^*_j| $ is defined as the minimum of the clockwise and
counterclockwise arc lengths between the phase angles $\theta^*_i,\theta^*_j\in\torus^1$. Let $G$ be an undirected weighted connected graph with edge set $ \mathcal{E}$
and let $\gamma\in[0,\pi)$ We define the cohesive subset $\Delta^G(\gamma)
  \subseteq [\torus^{n}]$ by
  \begin{align*}
    \Delta^G(\gamma) =\setdef{[\theta]\in[\torus^n]}{|\theta_i-\theta_j|\leq\gamma,
      \text{for all }(i,j)\in\mathcal{E}}.
  \end{align*}
  For every $\gamma\in[0,\pi)$, we define
    the \emph{embedded cohesive subset} $S^G(\gamma)\subseteq [\torus^{n}]$
    by:
\begin{equation*}
  S^G(\gamma)=\setdef{[\mathrm{exp}(\imagunit\mathbf{x})]}{\mathbf{x}\in
    B^G(\gamma), s\in[0,2\pi)},
\end{equation*}
where $B^G(\gamma) =
\setdef{\mathbf{x}\in\vectorones[n]^{\perp}}{\norm{\Bt\mathbf{x}}{\infty}\leq\gamma}$. Note
that, in general, we have $S^G(\gamma)\subseteq\Delta^G(\gamma)$. We refer to~\cite{SJ-FB:16h}
for additional properties of embedded cohesive subset. In particular,
it is shown that $S^G(\gamma)$ is diffeomorphic with $B^G(\gamma)$,
for every $\gamma\in [0,\frac{\pi}{2})$~\cite[Theorem 8]{SJ-FB:16h}. Using this result, in the
rest of this paper we identify the set $S^G(\gamma)$ with
$B^G(\gamma)$.

\section{Equivalent transcriptions of the equilibrium manifold}
\label{sec:auxiliary}

Consider an undirected graph $G$ with vertex set
$\mathcal{N}=\{1,\ldots,n\}$ and edge set $\mathcal{E} \subseteq
\mathcal{N}\times \mathcal{N}$ with $|\mathcal{E}| = m$. We start by
introducing three vector spaces defined by $G$:
\begin{enumerate}
\item the \emph{node space} is $\real^{n}$; elements of this space are
  called \emph{node vectors};
\item the \emph{edge space} is $\real^m$; elements of this space are called
  \emph{edge vectors}; and
\item the \emph{flow vector space} is $\Img(B^{\top})$; elements of
  $\real^m$ belonging to this space are called by \emph{flow vectors}.
\end{enumerate}
It is easy to see that an edge vector $\mathbf{z}\in \real^m$ is a flow
vector if and only if there exists a node vector $\mathbf{x}\in \real^{n}$
such that $\mathbf{z} = B^{\top}\mathbf{x}$.

Next, we introduce four different balance equations on an undirected graph
$G$ with incidence matrix $B$, weight matrix $\mcA$, cutset projection
$\prjcut$, and cycle projection $\prjcyc$. Given a node vector $\omega\in
\vect{1}^{\perp}_n$, define the shorthand flow vector $\eta=
B^{\top}L^{\dagger}\omega\in \Img(B^{\top})$.  The \emph{node balance
  equations} in the unknown node vector $\mathbf{x}\in \vect{1}_n^{\perp}$
is
\begin{align}\label{eq:nodal_sync}
  \omega = B\mcA\sin(B^{\top}\mathbf{x}).
\end{align}
The \emph{flow balance equations} in the unknown flow vector $\mathbf{z}\in
\Img(B^{\top})$ is
\begin{align}\label{eq:edge_sync}
  \eta = \prjcut \sin(\mathbf{z}).
\end{align} 
The \emph{constrained edge balance equations} in the unknown edge vector $\psi\in
\real^m$ is
\begin{align}\label{eq:aux_sync}
  \begin{cases}
    \eta = \prjcut \psi, \\ \arcsin(\psi) \in \Img(B^{\top}),\quad
    \|\psi\|_{\infty}\le 1.
  \end{cases}
\end{align}
The \emph{unconstrained edge balance equations} in the unknown edge vector $\phi\in
\real^m$ is
\begin{align}\label{eq:cut-flow}
  \eta = \prjcut \phi + \prjcyc \arcsin(\phi),\quad \|\phi\|_{\infty}\le 1.
\end{align}

We now present equivalent characterizations for synchronization manifold of
the Kuramoto model~\eqref{eq:kuramoto_model}.

\begin{theorem}[\textbf{Characterization of synchronization manifold}]
  \label{thm:existence_uniqueness_S}
 Consider an undirected connected graph $G$ with incidence matrix $B$, weight matrix
 $\mcA$, cutset projection $\prjcut$, and cycle projection $\prjcyc$. Given
 a node vector $\omega\in \vect{1}^{\perp}_n$, define the shorthand $\eta=
 B^{\top}L^{\dagger}\omega\in \Img(B^{\top})$. Pick an angle $\gamma\in
 [0,\frac{\pi}{2})$. Then the following statements are equivalent:
\begin{enumerate}
\item\label{p1:stable} there exists a unique locally exponentially stable
  synchronization manifold $\mathbf{x}^*$ for the Kuramoto
  model~\eqref{eq:kuramoto_model} in $S^{G}(\gamma)$;
\item \label{p2:nodal_balance} the node balance
  equations~\eqref{eq:nodal_sync} have a unique solution $\mathbf{x}^*$ in
  $S^{G}(\gamma)$;
\item \label{p3:edge_balance} the flow balance
  equations~\eqref{eq:edge_sync} have a unique solution $\mathbf{z}^*\in
  \Img(B^{\top})$ with $\|\mathbf{z}^*\|_{\infty}\le \gamma$;
\item \label{p4:auxiliary} the constrained edge balance
  equations~\eqref{eq:aux_sync} have a unique solution $\psi^*\in \real^m$
  with $\|\psi^*\|_{\infty}\le \sin(\gamma)$;
\item\label{p5:cut-flow} the unconstrained edge balance
  equations~\eqref{eq:cut-flow} have a unique solution $\phi^*\in \real^m$
  with $\|\phi^*\|_{\infty}\le \sin(\gamma)$.
\end{enumerate} 
Moreover, if one of the above equivalent conditions hold, then
\begin{align*}
\mathbf{z}^* &= B^{\top}\mathbf{x}^*,\quad\text{and }\quad
\psi^* = \phi^* = \sin(B^{\top}\mathbf{x}^*).
\end{align*}
\end{theorem}
\begin{proof}
The implications $\ref{p1:stable}\implies\ref{p2:nodal_balance}$
and $\ref{p2:nodal_balance}\implies\ref{p3:edge_balance}$ are easy
to show.

Regarding $\ref{p3:edge_balance}\implies\ref{p4:auxiliary}$, if
$\mathbf{z}^*\in \Img(B^{\top})$ is the unique solution to the flow balance
equations~\eqref{eq:edge_sync}, then $\phi^* = \sin(\mathbf{z}^*)\in
\real^m$ satisfies $\|\phi^*\|_{\infty}\le \sin(\gamma)$ and is a solution
for the edge balance equations~\eqref{eq:aux_sync}. Now, we show that
$\phi^*$ is the unique solution for the constrained edge balance
equations~\eqref{eq:aux_sync} such that $\|\phi^*\|_{\infty}\le
\sin(\gamma)$. Suppose that $\eta^*\ne \phi^*$ is another solution of the
constrained edge balance equations~\eqref{eq:aux_sync} satisfying
$\|\eta^*\|_{\infty}\le \sin(\gamma)$. Then, by the constrained edge balance
equations~\eqref{eq:aux_sync}, there exists $\mathbf{y}^*\in \Img(B^{\top})$ such
that $\mathbf{y}^*\ne \mathbf{z}^*$ and
$\arcsin(\eta^*)=\mathbf{z}^*$. This implies that $\|\mathbf{y}^*\| \le
\gamma$ and $\prjcut \sin(\mathbf{y}^*) =
B^{\top}L^{\dagger}\omega$. Therefore, $\mathbf{y}^*\in S^{G}(\gamma)$ and
satisfies the flow balance equations~\eqref{eq:edge_sync}. However, this is in
contradiction with the facts that $\mathbf{y}^*\ne \mathbf{z}^*$ and that
$\mathbf{z}^*\in\Img(B^{\top})$ is the unique solution of the flow balance
equations~\eqref{eq:edge_sync}.

Regarding $\ref{p4:auxiliary}\implies\ref{p5:cut-flow}$, if $\psi^*\in
\real^m$ is a solution of constrained edge balance equations~\eqref{eq:aux_sync} satisfying
$\|\psi^*\|_{\infty}\le \sin(\gamma)$, then
\begin{align}
B^{\top}L^{\dagger}\omega &= \prjcut \psi^*, \label{eq:first_aux}\\
\arcsin(\psi^*)&\in \Img(B^{\top}).\label{eq:second_aux}
\end{align}
Because $\Ker(\prjcyc)= \Img(B^{\top})$, the
inclusion~\eqref{eq:second_aux} implies that
\begin{align}\label{eq:third_aux}
  \prjcyc \arcsin(\psi^*) = 0. 
\end{align}
By adding equations~\eqref{eq:first_aux} and~\eqref{eq:third_aux}, we
obtain $B^{\top}L^{\dagger}\omega = \prjcut \psi^* + \prjcyc
\arcsin(\psi^*)$. This means that $\psi^*$ satisfies unconstrained edge
balance equations~\eqref{eq:cut-flow}.

Regarding $\ref{p5:cut-flow}\implies\ref{p1:stable}$, if $\phi^*\in\real^m$
solves the unconstrained edge balance equations~\eqref{eq:cut-flow}, then
\begin{align}\label{eq:fourth}
  B^{\top}L^{\dagger}\omega = \prjcut \phi^* + \prjcyc \arcsin(\phi^*).
\end{align}
Left-multiplying both sides of equations~\eqref{eq:fourth} by $\prjcut$ and using
the facts that $\prjcut B^{\top} = B^{\top}$, $\prjcut \prjcut = \prjcut$
and $\prjcut \prjcyc = \vect{0}_{m\times m}$, we obtain
\begin{align*}
  B^{\top}L^{\dagger}\omega = \prjcut \phi^*.
\end{align*}
Left-multiplying both side of the equations~\eqref{eq:fourth} by $\prjcyc$ we
obtain
\begin{align*}
  \prjcyc \arcsin(\phi^*)=\vect{0}_m.
\end{align*}
This last equality implies that $\arcsin(\phi^*)\in \Ker(\prjcyc) =
\Img(B^{\top})$. Thus, there exists a vector $\mathbf{x}^*\in
\vect{1}_n^{\perp}$ such that $\arcsin(\phi^*)=B^{\top}\mathbf{x}^*$. First,
note that $\prjcut \sin(B^{\top}\mathbf{x}^*) = B^{\top}L^{\dagger}\omega$
and, by multiplying both side of this equation by $B\mcA$, we obtain
\begin{align*}
  \omega = B\mcA\sin (B^{\top}\mathbf{x}^*).  
\end{align*}
Moreover, $\|\phi^*\|_{\infty}\le \gamma$. Thus, we have
$\|\arcsin(\phi^*)\|_{\infty} \le \gamma$ and
$\|B^{\top}\mathbf{x}^*\|_{\infty} \le \gamma$. This implies that
$\mathbf{x}^*\in \vect{1}_n^{\perp}$ is a synchronization manifold for the
Kuramoto model~\eqref{eq:kuramoto_model} in $S^{G}(\gamma)$. The uniqueness
follows from~\cite[Theorem~10, statement~(ii)]{SJ-FB:16h}.
\end{proof}

\section{Solvability of the unconstrained edge balance equations}\label{sec:cut-flow-balance-solution}


The unconstrained edge balance equations~\eqref{eq:cut-flow} allow us to
focus on a single analytic map whose inverse can be used in computing the
synchronization solutions of Kuramoto model. In this section, we study the
solvability of these equations and find their inverse on a suitable
domain. We start with relaxing the condition $\eta\in \Img(B^{\top})$ and
complexifing the equations~\eqref{eq:cut-flow}. This extension will allow us
to use the theory of several complex variables to find the Taylor series
expansion for the inversion of the complexified equations and prove the
strong convergence of the Taylor series. We then restrict back to real
domain and use the constrained $\eta\in\Img(B^{\top})$ to find the solutions
of the unconstrained edge balance equations~\eqref{eq:cut-flow}. We start with some useful definitions.  Given an undirected graph $G$ with
cutset projection $\prjcut$ and cycle projection $\prjcyc$, define the
\emph{complex edge balance map}
$\map{\FCmap}{\mathrm{D}^{\complex}(\vect{0}_m,\sin(\gamma))}{\complex^m}$
by
\begin{align*}
  \FCmap(\phi) = \prjcut \phi + \prjcyc\arcsin(\phi) 
\end{align*}
and the real edge balance map
$\map{\Fmap}{\mathrm{D}(\vect{0}_m,\sin(\gamma))}{\real^m}$ by
\begin{align*}
  \Fmap(\phi) = \prjcut \phi + \prjcyc\arcsin(\phi) .
\end{align*}
With this notation, the unconstrained edge balance
equations~\eqref{eq:cut-flow} read $\eta=\Fmap(\phi)$, together with the
constraints $\|\phi\|_{\infty}\le1$.

Next, we define the scalar function $\map{h}{\real_{\ge0}}{\real}$ by:
\begin{align*}
  h (x) = (x+1) \sqrt{1-\left(\frac{x}{x+1}\right)^2} - x \arccos\left(\frac{x}{x+1}\right).
\end{align*} 
The graph of function $h$ on the interval $[0,20]$ is shown in
Figure~\eqref{fig:function-h}.
\begin{figure}[!htb]\centering
  \includegraphics[width=.65\linewidth]{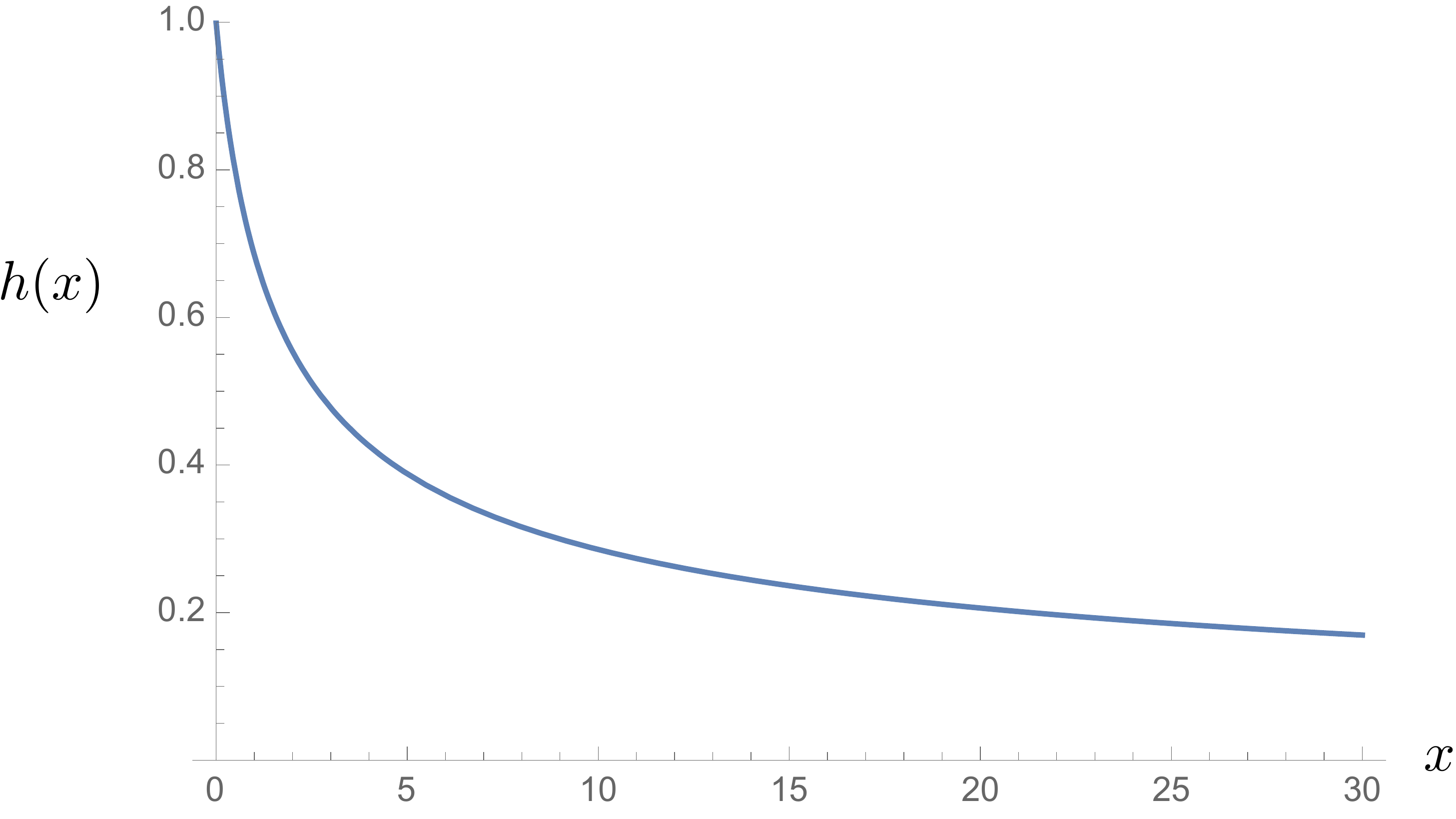}
  \caption{The graph of the monotonically-decreasing function $h$}
  \label{fig:function-h}
\end{figure}
Since $h$ is continuous and strictly monotonically-decreasing, its inverse
exists and is denoted by $\map{h^{-1}}{\real}{\real_{\geq 0}}$. Although we
do not have an analytical form for $h^{-1}(y)$, it is simple to compute
numerically.

We are now ready to provide an estimate on the image of the maps
$\FCmap$ and $\Fmap$ and to present a power series expansion for the
inverse maps $\FCmap^{-1}$ and $\Fmap^{-1}$ on suitable domains.

\begin{theorem}[\textbf{Properties of the complex edge balance map}]
  \label{thm:property-G-complex}
 Consider an undirected connected graph $G$ with cutset projection $\prjcut$ and
 cycle projection $\prjcyc$.  Select $\eta\in \real^m$ such that
 $\|\eta\|_{\infty}<h(\|\prjcyc\|_{\infty})$ and define $\gamma^*\in
 [0,\tfrac{\pi}{2})$ by
   \begin{align*}
     \gamma^*= \arccos\left(\frac{h^{-1}(\|\eta\|_{\infty})}{h^{-1}(\|\eta\|_{\infty}) + 1}\right).
   \end{align*}
Then the following statements holds:
\begin{enumerate}
\item\label{p1:existence-G} there exists a unique $\phi^*\in
  \mathrm{D}^{\complex}(\vect{0}_m,\sin(\gamma^*))$ such that
  $\FCmap(\phi^*)=\eta$; that is unconstrained edge balance equations have a
  unique solution;
\item\label{p3:inverse-G} there exists a holomorphic map
$\FCmap^{-1}:
\mathrm{D}^{\complex}(\vect{0}_m,\|\eta\|_{\infty}) \to \FCmap^{-1}(\mathrm{D}^{\complex}(\vect{0}_m,\|\eta\|_{\infty}))$
such that 
\begin{align*}
\FCmap^{-1}\scirc \FCmap(\phi) &= \phi
  ,\qquad\text{for all } \phi\in \FCmap^{-1}(\mathrm{D}^{\complex}(\vect{0}_m,\|\eta\|_{\infty})),\\
\FCmap \scirc \FCmap^{-1} (\xi) &=
  \xi,\qquad\text{for all } \xi \in \mathrm{D}^{\complex}(\vect{0}_m,\|\eta\|_{\infty});
\end{align*}
that is the edge balance map is invertible on $\mathrm{D}^{\complex}(\vect{0}_m,\|\eta\|_{\infty})$; 
\item\label{p4:power_series-G} the power series
  \begin{align*}
    \sum_{i=0}^{\infty} A_{2i+1}(\eta)=A_1(\eta)+A_3(\eta)+A_5(\eta)+\ldots,
  \end{align*}
  converges strongly to $\FCmap^{-1}(\eta)$, where, for every $i\in
  \Z_{\ge 0}$, the term $A_i(\eta)$ is a homogeneous polynomial of order
  $i$ in $\eta$
   defined iteratively by:
\begin{align*}
A_1(\eta) & = \eta, \nonumber \\
A_{2i+1}(\eta) & = -\prjcyc \Bigg(  \sum_{k=1}^i \frac{(2k-1)!!}{(2k)!!(2k\!+\!1)} 
\sum_{ \substack{\text{odd } \alpha_1,\dots,\alpha_{2k+1}  \text{ s.t.}\\
		\alpha_1+\dots+\alpha_{2k+1}=2i+1}} \!\!\!\!\!\!\!\!\!\!\!
A_{\alpha_1}(\eta)\circ \cdots \circ A_{\alpha_{2k+1}}(\eta)\Bigg). 
\end{align*}
\end{enumerate}
\end{theorem}
\begin{proof}
Regarding statement~\ref{p1:existence-G}, define the map
$\map{H_{\eta}}{\mathrm{D}^{\complex}(\vect{0}_m,\sin(\gamma^*))}{\real^m}$
by
\begin{align*}
  H_{\eta}(\phi) = \eta - \prjcyc(\arcsin(\phi) -\phi ).
\end{align*}
The map $H_{\eta}$ appears from bringing all terms of the unconstrained edge
equation~\eqref{eq:cut-flow} to the left hand side and adding $\phi$ to
both sides.

First, we show that
$H_{\eta}(\mathrm{D}^{\complex}(\vect{0}_m,\sin(\gamma^*))) \subseteq
\mathrm{D}^{\complex}(\vect{0}_m,\sin(\gamma^*))$. For $\phi\in
\mathrm{D}^{\complex}(\vect{0}_m,\sin(\gamma^*))$, we compute
\begin{align*}
\|H_{\eta}(\phi)\|_{\infty} = \|\eta + \prjcyc(\arcsin(\phi) -\phi
  )\|_{\infty}  \le \|\eta\|_{\infty} + \|\prjcyc\|_{\infty}
  \|\phi- \arcsin(\phi)\|_{\infty}. 
\end{align*}
Moreover, for $\phi\in \mathrm{D}^{\complex}(\vect{0}_m,\sin(\gamma^*))$,
we have $\|\phi- \arcsin(\phi)\|_{\infty} \le \gamma^* -
\sin(\gamma^*)$. These equalities imply that
\begin{multline}\label{eq:important}
\|H_{\eta}(\phi)\|_{\infty}\le \|\eta\|_{\infty} + \|\prjcyc\|_{\infty}
  (\gamma^* - \sin(\gamma^*)) \\ \le \|\eta\|_{\infty} + h^{-1}(\|\eta\|_{\infty})
  (\gamma^* - \sin(\gamma^*)),
\end{multline}
where, for the last inequality, we used the fact that $\|\prjcyc\|_{\infty}
\le h^{-1}(\|\eta\|_{\infty})$. By the definition of $h$, we have
\begin{multline*}
\|\eta\|_{\infty} =
  (h^{-1}(\|\eta\|_{\infty})+1)\sqrt{1-\left(\frac{h^{-1}(\|\eta\|_{\infty})}{h^{-1}(\|\eta\|_{\infty})+1}\right)}
  \\ - h^{-1}(\|\eta\|_{\infty}) \arccos\left(\frac{h^{-1}(\|\eta\|_{\infty})}{h^{-1}(\|\eta\|_{\infty})+1}\right).
\end{multline*}
Noting the fact that
$\arccos\left(\frac{h^{-1}(\|\eta\|_{\infty})}{h^{-1}(\|\eta\|_{\infty})+1}\right)
= \gamma^*$, we obtain
\begin{align*}
 \|\eta\|_{\infty} = (h^{-1}(\|\eta\|_{\infty})+1)\sin(\gamma^*) - h^{-1}(\|\eta\|_{\infty})\gamma^*.
\end{align*}
Now, by replacing the above equation into
inequality~\eqref{eq:important}, we have
\begin{align*}
\|H_{\eta}(f)\|_{\infty}& \le  \|\eta\|_{\infty} + h^{-1}(\|\eta\|_{\infty})
  (\gamma^* - \sin(\gamma^*)) \\ & =
  (h^{-1}(\|\eta\|_{\infty})+1)\sin(\gamma^*) -
  h^{-1}(\|\eta\|_{\infty})\gamma^* + h^{-1}(\|\eta\|_{\infty})
  (\gamma^* - \sin(\gamma^*)) \\ &= \sin(\gamma^*).
\end{align*}
Thus, by the Banach Fixed-Point Theorem, there exists a unique fixed point
$\phi^*\in \mathrm{D}^{\complex}(\vect{0}_m,\sin(\gamma^*))$ for
$H_{\eta}$.  By construction, this fixed element $\phi^*\in
\mathrm{D}^{\complex}(\vect{0}_m,\sin(\gamma^*))$ satisfies
\begin{align*}
  \eta =  \prjcyc \arcsin(\phi^*) + \prjcut \phi^* = \FCmap (\phi^*).
\end{align*}
This completes the proof of statement~\ref{p1:existence-G}.

Regarding statement~\ref{p3:inverse-G}, by statement~\ref{p1:existence-G}, for every
$\eta\in \complex^m$ such that $\|\eta\|_{\infty}<
h(\|\prjcyc\|_{\infty})$, there exists a unique $\phi\in \real^m$ such
that $\eta = \FCmap(\phi)$. This implies that $\FCmap$ has a unique
inverse $\FCmap^{-1}:\FCmap(\mathrm{D}^{\complex}(\vect{0}_m,\|\eta\|_{\infty}))\to\mathrm{D}^{\complex}(\vect{0}_m,\|\eta\|_{\infty})$
which satisfies the equalities in statement~\ref{p3:inverse-G}. Now we show that $\FCmap^{-1}$ is holomorphic on
$\FCmap(\mathrm{D}^{\complex}(\vect{0}_m,\|\eta\|_{\infty}))$. Note
that, for every $\phi\in \mathrm{D}^{\complex}(\vect{0}_m,\|\eta\|_{\infty})$, the derivative
of the map $\FCmap$ at point $\phi$ is given by: 
\begin{align*}
D_{\phi}\FCmap = \prjcut+ \prjcyc \diag\Big(\frac{1}{\sqrt{1-\phi_i^2}}\Big).
\end{align*}
We first show that $D_{\phi}\FCmap $ is invertible. 
Suppose that, there exists $\mathbf{x}\in
\real^m$ such that $D_{\phi}\FCmap(\mathbf{x})=0$. This means that $\prjcut \mathbf{x} = \vect{0}_m$ and $\prjcyc
\diag\Big(\frac{1}{\sqrt{1-\phi_i^2}}\Big)\mathbf{x} = \vect{0}_m$. The first
equality implies that $\mathbf{x}\in \Ker(B\mathcal{A})$ and the second
inequality implies that $\diag\Big(\frac{1}{\sqrt{1-\phi_i^2}}\Big)\mathbf{x}
\in \Img(B^{\top})$. Therefore, there exists $\alpha\in
\vect{1}_n^{\perp}$ such that
$\diag\Big(\frac{1}{\sqrt{1-\phi_i^2}}\Big)\mathbf{x} = B^{\top}\alpha$. Thus,
we get 
\begin{align}\label{eq:laplacian}
B \mathcal{A} \diag(\sqrt{1-\phi_i^2}) B^{\top}\alpha = B \mathcal{A} \mathbf{x} = \vect{0}_n.
\end{align}
Moreover, $\mathcal{A} \diag(\sqrt{1-\phi_i^2})$ is a diagonal matrix with positive
diagonal elements. Therefore, equations~\eqref{eq:laplacian} implies that $\alpha \in \mathrm{span}\{\vect{1}_n\}$ and as a result $\mathbf{x} = \diag(\sqrt{1-\phi_i^2})
B^{\top}\alpha = \vect{0}_n$. This proves that the derivative $D_{\phi}\FCmap$ is
invertible. Now, by the Inverse Function Theorem~\cite[Theorem 2.5.2]{RA-JEM-TSR:88}, the maps
$\FCmap $ and $\FCmap^{-1}$ are locally holomorphic and therefore they
are holomorphic on their domains. This completes
the proof of statement~\ref{p3:inverse-G}.

Regarding statement~\ref{p4:power_series-G}, we first find the formal power
series representation for $\FCmap^{-1}$. Suppose that
$\sum_{i=1}^{\infty} A_{i}(\xi)$ is the formal power series for
$\FCmap^{-1}$. Then we have
\begin{align*}
\prjcyc \arcsin(\FCmap^{-1}(\eta)) + \prjcut
  \FCmap^{-1}(\eta) = \eta, \qquad \mbox{ for all } \eta\in \real^m. 
\end{align*}
By replacing the power series $\sum_{i=1}^{\infty} A_{i}(\xi)$ for
$\FCmap^{-1}$ and using the power series expansion of $\arcsin$, we obtain 
\begin{align}\label{eq:series_equality}
\prjcyc \left(\sum_{k=1}^{\infty} \frac{(2k-1)!!}{(2k)!!(2k+1)}
  \big(\sum_{i=1}^{\infty} A_i(\eta)\big)^{\circ (2k-1)} \right)+ \prjcut
  \sum_{i=1}^{\infty} A_i(\eta) = \eta. 
\end{align}
By equating the same order terms on the both side of
equation~\eqref{eq:series_equality} and using the fact that $\prjcyc +
\prjcut=I_m$, we obtain that $A_{1}(\eta) = \eta$ and $A_{2i}(\eta) = \vect{0}_m$, for every $i\in
\Z_{\ge 0}$. Simple book-keeping shows that the recursive
formula in statement~\ref{p4:power_series-G} holds for the odd terms in the power series. 

Finally, we prove that the formal power series
$\sum_{i=1}^{\infty}A_{2i+1}(\eta)$ converges on the domain
$\mathrm{D}^{\complex}(\vect{0}_m,\eta)$. Note that
statement~\ref{p3:inverse-G} implies that the map
$\map{\FCmap^{-1}}{\mathrm{D}^{\complex}(\vect{0}_m,\eta)}{\FCmap^{-1}(\mathrm{D}^{\complex}(\vect{0}_m,\eta))}$
is holomorphic and and that the set
$\mathrm{D}^{\complex}(\vect{0}_m,\eta)$ is a Reinhardt domain. Therefore,
\cite[Theorem 2.4.5]{LH:90} implies that the power series converges
strongly on the domain $\mathrm{D}^{\complex}(\vect{0}_m,\eta)$.
\end{proof}

\begin{remark}[\textbf{Properties of the real edge balance map}]
  A similar result as Theorem~\ref{thm:property-G-complex} holds for the
  real edge balance map $\Fmap$ by replacing the complex variables by their
  real counterparts. The proof is straightforward by restricting the
  results in Theorem~\ref{thm:property-G-complex} to the real Euclidean
  space.
\end{remark}

\section{Inverse Taylor expansion for Kuramoto
  model}\label{sec:syncOfKuramoto}

In this section we study the synchronization of the Kuramoto
model~\eqref{eq:kuramoto_model} by applying the results on the
unconstrained edge balance equations~\eqref{eq:cut-flow} from
Theorem~\ref{thm:property-G-complex} in the previous section.

\begin{theorem}[\textbf{Inverse Taylor expansion}]
  \label{thm:suff_condition+power_series}
Consider the Kuramoto model~\eqref{eq:kuramoto_model} with undirected connected graph $G$,
weighted cutset projection $\prjcut$, and weighted cycle projection
$\prjcyc$. Given frequencies $\omega\in \vect{1}_n^{\perp}$ satisfying
\begin{align}\label{test:powerSeries}\tag{T0}
\|B^{\top}L^{\dagger}\omega\|_{\infty} < h(\|\prjcyc\|_{\infty}),
\end{align} 
define $\gamma^*\in[0,\tfrac{\pi}{2})$ by
\begin{align*}
  \gamma^* = \arccos
  \left(\frac{h^{-1}(\|B^{\top}L^{\dagger}\omega\|_{\infty})}{ h^{-1}(\|B^{\top}L^{\dagger}\omega\|_{\infty})+ 1}\right).
\end{align*}
Then the following statements hold:
\begin{enumerate}
\item\label{p1:ex} there exists a unique locally stable synchronization
  manifold $\mathbf{x}^*$ in $S^{G}(\gamma^*)$; and
\item\label{p3:series} the power series
  \begin{align}\label{eq:inverse_power_series}
    \sum_{i=0}^{\infty} A_{2i+1}(B^{\top}L^{\dagger}\omega) =
    A_1(B^{\top}L^{\dagger}\omega)+A_3(B^{\top}L^{\dagger}\omega)+\ldots,
  \end{align}
  converges strongly to $\sin(B^{\top}\mathbf{x}^*)$ where, for every $i\in
  \Z_{\ge 0}$, the term $A_i(\eta)$ is a homogeneous polynomial of
  order $i$ in $\eta$ defined iteratively as in Theorem~\ref{thm:property-G-complex}\ref{p4:power_series-G}.
\end{enumerate}
\end{theorem}

This theorem is an immediate application of
Theorem~\ref{thm:existence_uniqueness_S} on the equivalent transcriptions
and of Theorem~\ref{thm:property-G-complex} on the properties of the maps
$\FCmap$ and $\Fmap$.
  
\begin{proof}[Proof of Theorem~\ref{thm:suff_condition+power_series}]
  Regarding statement~\ref{p1:ex}, we use (the real version of)
  Theorem~\ref{thm:property-G-complex}\ref{p1:existence-G} with
  $\eta=B^{\top}L^{\dagger}\omega$. Since
  $\|B^{\top}L^{\dagger}\omega\|_{\infty}<h(\|\prjcyc\|_{\infty})$. Therefore,
  there exists a unique $\phi^*\in \mathrm{D}(\vect{0}_m,\sin(\gamma^*))$
  such that $\Fmap(\phi) = B^{\top}L^{\dagger}\omega$. This means that
  \begin{align*}
    B^{\top}L^{\dagger}\omega =  \prjcyc \arcsin(\phi^*) + \prjcut \phi^*.
  \end{align*}
  Since $\Img(B^{\top}) = \Img(\prjcut) = \Ker(\prjcyc)$, we obtain $\prjcut
  \phi^* = B^{\top}L^{\dagger}\omega$ and $\arcsin(\phi)^*\in
  \Img(B^{\top})$. The result follows by the equivalence of
  parts~\ref{p1:stable} and~\ref{p4:auxiliary} in
  Theorem~\ref{thm:existence_uniqueness_S}.  
  
  Regarding statement~\ref{p3:series}, the result follows from (the complex
  version of) Theorem~\ref{thm:property-G-complex}\ref{p4:power_series-G}.
\end{proof}

Some remarks are in order.

\begin{remark}[\textbf{Power series expansion for $\sin(B^{\top}\mathbf{x}^*)$}]
  \label{remark:series}
\begin{enumerate}
\item It is instructive to apply the iterative procedure in
  Theorem~\ref{thm:suff_condition+power_series}\ref{p3:series} to compute
  the first four odd terms in the power
  series~\eqref{eq:inverse_power_series} where $
  \eta=B^{\top}L^{\dagger}\omega $:
  \begin{align*}
    &A_1(\eta)=\eta,\\
    &A_3(\eta)=-\prjcyc \bigg( \frac{1}{6}\eta^{\circ 3} \bigg),\\
    &A_5(\eta)=-\prjcyc \bigg(\frac{1}{12} A_3(\eta)\scirc \eta^{\circ 2}+\frac{3}{40}\eta^{\circ 5} \bigg),\\
    &A_7(\eta)=-\prjcyc \bigg( \frac{5}{112}\eta^{\circ 7} + \frac{3}{8}A_3(\eta)\scirc\eta^{\circ 4} + \frac{1}{2}(A_3(\eta))^{\circ 2}\scirc\eta + \frac{1}{2}A_5(\eta)\scirc\eta^{\circ 2} \bigg).
  \end{align*}
  The iterative procedure in
  Theorem~\ref{thm:suff_condition+power_series}\ref{p3:series} is amenable
  to implementation on a mathematical software manipulation system; we
  report its implementation in Mathematica code in
  Algorithm~\ref{algorithm_series} in Appendix~\ref{app:mathematics_code}.

\item If $\prjcyc$ and $\diag(\eta)$ commute, then, for every $i\in
  \Z_{>0}$:
  \begin{align*}
    A_{2i+1}(\eta) = -\frac{(2i-1)!!}{(2i)!!(2i+1)} \prjcyc\left(\eta\right)^{\circ (2i+1)}.
  \end{align*}
  For example, if the graph $G$ is acyclic, then $\prjcyc =
  \vect{0}_{n\times n}$ and, therefore, $\prjcyc$ and $\diag(\eta)$
  commute. Thus, for acyclic graphs, we have
  \begin{align*}
    A_{1}(\eta) &= \eta, \\ 
    A_{2i+1}(\eta) &= 0 , \qquad\text{for all } i\in \Z_{\ge 0}
  \end{align*}
  so that $\sin(B^{\top}\mathbf{x}^*) =
  B^{\top}L^{\dagger}\omega$. Therefore, the Kuramoto
  model~\eqref{eq:kuramoto_model} on an acyclic graph has a unique locally
  stable synchronization manifold inside $S^G(\gamma)$ if and only if
  $\|B^{\top}L^{\dagger}\omega\|_{\infty}\le \sin(\gamma)$. Moreover, if
  this condition holds, then the synchronization manifold is given by
  \begin{align*}
    \mathbf{x}^* =
    L^{\dagger}B\mathcal{A}\arcsin(B^{\top}L^{\dagger}\omega).
  \end{align*}
  This result is known for example as~\cite[Theorem~2 (Supporting
    Information)]{FD-MC-FB:11v-pnas}.
  
\item While for acyclic graphs we have $\prjcyc =\vect{0}_{n\times n}$, the
  matrix $\prjcyc$ is non-zero and idempotent for cyclic graphs and it satisfies
  $\|\prjcyc\|_{\infty}\ge 1$. The jump from $\norm{\prjcyc}{\infty}$ equals $0$ to values greater than or equal to $1$ can be attributed to the
  discontinuity of the projection matrix $\prjcyc$ with
  respect to edge weights of the graph. The following example shows that
  the infinity norm of the projection matrix $\prjcyc$ is, in general, a
  discontinuous function of the weights of the graphs. Consider the family
  of $3$-cycle graph $\{G(\epsilon)\}_{\epsilon\ge 0}$ with the node set
  $V=\{1,2,3\}$, the edge set $\mathcal{E} =\{(1,2), (1,3), (2,3)\}$, and
  the adjacency matrix $A(\epsilon)\in \real^{3\times 3}$ given by
  \begin{align*}
    A(\epsilon) = \begin{pmatrix}
      0 & 1 & 1 \\
      1 & 0 & \epsilon \\
      1 & \epsilon & 0
    \end{pmatrix}.
  \end{align*}
  Then, for every $\epsilon>0$ , one can show that
  \begin{align*}
    \prjcyc(\epsilon) = I_3 - B^{\top}L^{\dagger}B\mcA = \begin{pmatrix}
      \epsilon & \epsilon & \epsilon \\
      \epsilon & \epsilon & \epsilon \\
      1-2\epsilon & 1-2\epsilon & 1-2\epsilon 
    \end{pmatrix}.
  \end{align*}
  This implies that $\lim_{\epsilon\to 0^{+}}
  \|\prjcyc(\epsilon)\|_{\infty} =3$. However, the graph $G(0)$ is acyclic
  and therefore we have $\prjcyc(0) = \vect{0}_{3\times 3}$. Thus,
  $\lim_{\epsilon\to 0^{+}} \|\prjcyc\|_{\infty}\ne
  \|\prjcyc(0)\|_{\infty}$. This implies that the function $\epsilon
  \mapsto \|\prjcyc(\epsilon)\|_{\infty}$ is not continuous at
  $\epsilon=0$.
\end{enumerate}
\end{remark}

In the rest of this section, we use the power series~\eqref{eq:inverse_power_series} to
propose a family of statistically-accurate approximate tests that are much
less conservative than proven sufficient conditions for existence of a
unique synchronization solution inside $ S^G(\gamma) $ for $
\gamma\in[0,\pi/2) $.  Our approximate tests estimate the solution of
the unconstrained edge balance equations~\eqref{eq:cut-flow} and check that all elements of the estimate are less than or equal
  to $\sin(\gamma)$.  For more insight into these approximate tests, recall
  from Theorem~\ref{thm:suff_condition+power_series}\ref{p3:series}
  that $\phi$, the solution for  unconstrained edge balance
  equations~\eqref{eq:cut-flow} is given by
\begin{align*}
  \phi=B^{\top}L^{\dagger}\omega+A_3(B^{\top}L^{\dagger}\omega)+A_5(B^{\top}L^{\dagger}\omega)+\ldots.
\end{align*}
There already exists a first order approximate synchronization test,
introduced by reference \cite{FD-MC-FB:11v-pnas}, which truncates the
series, given above, after the first order term. By
approximating $\phi$ with $\phi\approx
B^{\top}L^{\dagger}\omega\in\Img(B^{\top})$, we write
\begin{align*} 
  \tag{AT1}\label{test:order1}
  \big\|B^{\top}L^{\dagger}\omega\big\|_{\infty}\le \sin(\gamma),
\end{align*} 
By substituting the third order power series expansion for the edge
variable $\phi\in\real^m$ of the Kuramoto model in Theorem
\ref{thm:suff_condition+power_series}\ref{p3:series} into
equation~\eqref{eq:aux_sync}, we can also write the third order approximate
synchronization test as
\begin{align*}
  \Big\|  B^{\top}L^{\dagger}\omega +
    \frac{1}{6}\prjcyc(B^{\top}L^{\dagger}\omega)^{\circ3} \Big\|_{\infty}
  \leq\sin(\gamma).
\end{align*}
In summary we propose a family of higher order approximate tests as
follows.

\begin{definition} \label{def:approxTest}
  For $\gamma\in[0,\pi/2)$ and $k\in2\Z_{\geq0}+1$, the \emph{$k$th
      order approximate test} for existence of a unique solution in
    $S^{G}(\gamma)$ is defined by
    \begin{align*}\tag{AT$k$}\label{test:approx}
    \Biggnorm{\sum_{i=0}^{(k-1)/2} A_{2i+1}(B^{\top}L^{\dagger}\omega)}{\infty} \leq \sin(\gamma).
    \end{align*}
\end{definition}

\section{Numerical Experiments}\label{sec:numerical} 
In this section we illustrate the usefulness of the Taylor series expansion
given by Theorem~\ref{thm:suff_condition+power_series}. First, for large
IEEE test cases, we illustrate the accuracy of the truncated series for
approximating synchronized solutions of~\eqref{eq:kuramoto_model}. In
addition, we present results showing the sharpness of the approximate
tests~\eqref{test:approx} for existence of a synchronization manifold
\eqref{def:approxTest} on several IEEE test cases and random networks.

\subsection{Accuracy of the Taylor series: approximating the synchronization manifold}\label{sec:num-accuracy1}
Here we evaluate the accuracy of the truncated power series in
Theorem~\ref{thm:suff_condition+power_series}\ref{p3:series} for
approximating the synchronization manifold.  We consider both IEEE test
cases and random networks to evaluate these measures of accuracy.

The general numerical setting for the IEEE test cases is as follows.  Each
IEEE test case can be described by a connected undirected graph $ G $ with
the nodal admittance matrix $ Y\in\complex^{n\times n} $. The set of nodes
in $ G $ are partitioned into load buses $ \mathcal{N}_1 $ and generator
buses $ \mathcal{N}_2 $. The power demand (resp. power injection) at node $
i\in\mathcal{N}_1 $ (resp. $ i\in\mathcal{N}_2 $) is denoted by $ P_i $. $
V_i $ and $ \theta_i $ are the voltage magnitude and phase angle at node $
i\in\mathcal{N}_1\cup\mathcal{N}_2 $. For every IEEE test case, we study the following Kuramoto
synchronization manifold equation
\begin{equation}\label{eq:kuramoto-ieee}
  P_i=\sum_{j\in\mathcal{N}_1\cup\mathcal{N}_2}a_{ij}\sin(\theta_i-\theta_j), \qquad
  \text{for all } i\in\mathcal{N}_1\cup\mathcal{N}_2, 
\end{equation} 
where $ a_{ij} = a_{ji} = V_iV_j\Im(Y_{ij})>0 $. The
equations~\eqref{eq:kuramoto-ieee} are exactly the lossless active AC power
flow equations for the network. Note that in order to study
equations~\eqref{eq:kuramoto-ieee}, we need to apply some modifications to
the IEEE test cases. First, the admittance matrix $ Y \in j\real^{n\times
  n} $ is purely inductive with no shunt admittances $ Y_{ii}=0 $. If the
IEEE test case has branch resistances or shunt admittances, then they are
removed. Second, we assume that all the nodes in the IEEE test case are
$PV$ nodes; this assumption is reasonable since the active power injection
and output voltage of generators are known. For the loads, we use
MATPOWER~\cite{RDZ-CEMS-RJT:11} to solve the coupled AC power flow balance
to obtain their terminal voltage $V_i$. Lastly, for every $i\in
\{1,\ldots,n\}$, we set $ P_i = KP_i^{\mathrm{nom}} $ for some $
K\in\real_{\geq0} $ where $ P_i^{\mathrm{nom}} $ is the nominal injections
given by each test case. Starting with $ K=0 $, we increase $ K $ by $
5\times 10^{-3}$ at each step. MATLAB's \emph{fsolve} is used to solve
equations~\eqref{eq:kuramoto-ieee} for $\theta^*_{\mathrm{fsolve}}$ at each
$ K $. The scalar $ K $ is increased until whichever situation occurs
first: $ \|B^{\top}\subscr{\theta^*}{fsolve}\|_\infty $ reaches $ \pi/2 $
or \emph{fsolve} does not converge to a solution.

To evaluate the accuracy of the truncated Taylor series with $ k $ terms,
we define the absolute error denoted by $ S_k $ by
\begin{equation}\label{eq:series_error}
  S_k = \norm{\sin(B^{\top}\theta^*_{\mathrm{fsolve}}) -  \sum_{i=0}^{k}A_{2i+1}(B^{\top}L^{\dagger}\subscr{p}{sd})}{\infty},
\end{equation}
where $\subscr{p}{sd}=[P_1,\dots,P_n]^{\top}$ is the balanced supply/demand
vector, $L = B \mathcal{A} B^{\top}$, and $\mathcal{A}$ is the diagonal
weight matrix matrix with diagonal elements $\{a_{ij}\}_{(i,j)\in
  \mathcal{E}}$. The errors $ S_k$ for IEEE 300 and Pegase 1354 are shown
in Figure~\ref{fig:matPow_increaseP}.

\begin{figure}[!htb]
  \centering
  \includegraphics[width=\linewidth]{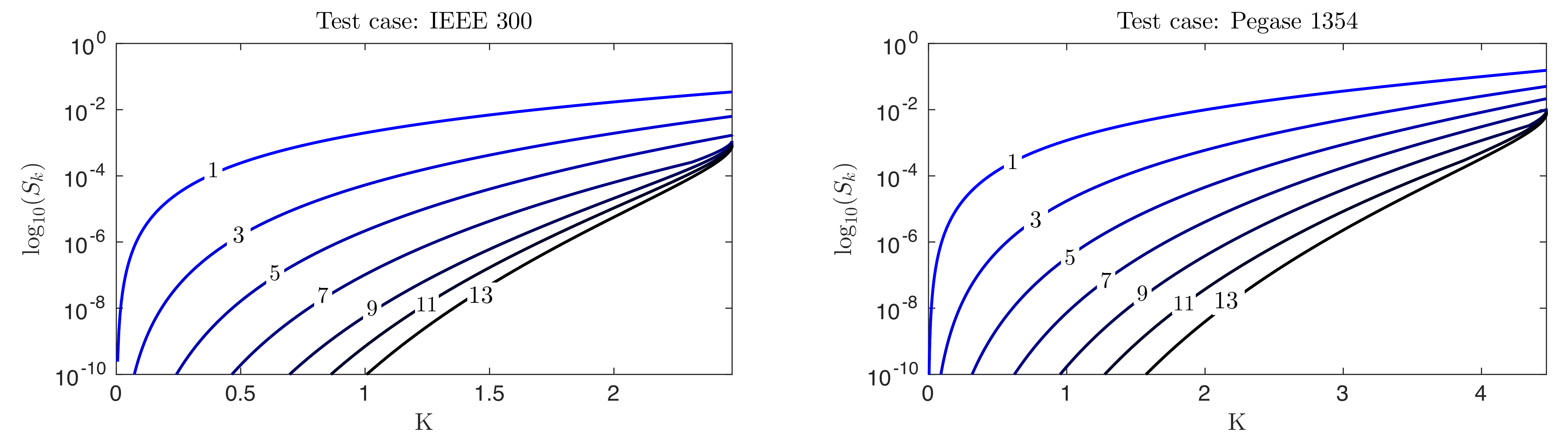}
  \caption{Comparison of the absolute errors of the sine of the phase
    differences approximated with the truncated Taylor series in
    Theorem~\ref{thm:suff_condition+power_series}\ref{p3:series} for all
    odd orders up to 13. The absolute error is calculated with equation
    \eqref{eq:series_error}, using the solution found with \emph{fsolve} as
    the true value. } \label{fig:matPow_increaseP}
\end{figure}

\paragraph{Summary evaluation} Figure~\ref{fig:matPow_increaseP} shows
that, for the IEEE test cases, the error of truncated Taylor series for computing the synchronized
solutions of the Kuramoto model~\eqref{eq:kuramoto-ieee} decreases exponentially with
the order of the truncations and increases as we approach
the threshold of synchronization. For IEEE 300 and Pegase 1354 with the nominal power
injections, the error of approximating the synchronized manifold with
$5$th order or higher truncated series is smaller than $10^{-6}$.

\subsection{Accuracy of the Taylor series: approximating the critical coupling}
In this section, we compare the approximate synchronization
test~\eqref{test:approx} with the existing tests in the literature and
evaluate the accuracy of these approximate tests. We consider both IEEE test
cases and random networks to evaluate these measures of accuracy.

For IEEE test cases, we use the same simulation setup given earlier in
Section~\ref{sec:num-accuracy1}. Denote the critical coupling of
equation~\eqref{eq:kuramoto-ieee} by $\subscr{K}{C}$, that is, let
$\subscr{K}{C}$ be the smallest scaling factor such that
$\|B^{\top}\subscr{\theta^*}{fsolve}\|_\infty $ reaches $ \pi/2 $ or
\emph{fsolve} does not converge to a solution (whichever occurs first).
Let $ \subscr{K}{T} $ denote the smallest scaling factor for which a
particular synchronization test fails. Then we denote the critical ratio by
$ \subscr{K}{T}/\subscr{K}{C} $; this percentage is a measure of the
accuracy of the given test. The conditions are checked with a
$10^{-6}$ tolerance.  Table \ref{tab:IEEE-test-cases} compares the
accuracy of the approximate test with existing sufficient conditions for
synchronization:

\emph{The first two columns} contain the critical ratio of two known
tests from the literature~\eqref{eq:test-eigval2} and~\eqref{eq:test-infty-norm} from~\cite[Theorem 7.2]{FD-FB:13b} and
\cite[Theorem 16]{SJ-FB:16h} respectively:
\begin{align} \tag{T1}
  \lambda_2(L) &> \lambda_\mathrm{critical} \triangleq \| B^{\top} \subscr{p}{sd}\|_2 \label{eq:test-eigval2},
  \\
  \tag{T2}\label{eq:test-infty-norm} 
  \| B^{\top}L^{\dagger} \subscr{p}{sd} \|_\infty &\leq g(\| \prjcut \|_\infty).
\end{align}
Note that test~\eqref{eq:test-infty-norm} is a sufficient condition for existence of a synchronization manifold in $ S^G(\gamma^*) $ where $ \gamma^*=\arccos\Big(\frac{\norm{\subscr{\mcP}{cut}}{\infty}-1}{\norm{\subscr{\mcP}{cut}}{\infty}+1}\Big)\in[0,\pi/2] $ and $ g(x) = \frac{y(x)+\sin(y(x))}{2} - x \frac{y(x)-\sin(y(x))}{2} \bigg\rvert_{y(x)=\arccos(\frac{x-1}{x+1})} $.

\emph{The third column} contains the new sufficient test \eqref{test:powerSeries} proposed in this paper.

\emph{The last four columns} contains the critical ratio for the family of approximate tests~\eqref{test:approx} of order $ 1 $, $ 3 $, $ 5 $ and $ 7 $.

\begin{center}
	\begin{table}[htb] 
		\resizebox{0.99\linewidth}{!}{\setlength{\tabcolsep}{3pt}%
			\begin{tabular}{|l|c|c|c|c|c|c|c|}
				\hline
				\multirow{4}{*}{Test Case}  
				&\multicolumn{7}{c|}{Critical ratio $\subscr{K}{T}/\subscr{K}{C}$}\\
				\cline{2-8}
				& & $\infty$-norm &  & Approx~test & Approx~test & Approx~test & Approx~test
				\\
				&$\lambda_2$ test & test&New test &$k=1$ &$k=3$ &$k=5$ &$k=7$\\
				&\eqref{eq:test-eigval2}\cite{FD-FB:13b}&\eqref{eq:test-infty-norm}\cite{SJ-FB:16h}&\eqref{test:powerSeries}%
				& \eqref{test:order1}\cite{FD-MC-FB:11v-pnas}&\eqref{test:approx}&\eqref{test:approx}&\eqref{test:approx}\\
				\hline
				%
				%
				%
				%
				%
				
				\rowcolor{Gray}
				IEEE 118 &0.23~$\%$&43.76~$\%$&29.91~$\%$&86.12~$\%$&90.80~$\%$&93.10~$\%$&94.45~$\%$\\
				
				IEEE 300 &0.02~$\%$&40.45~$\%$&27.25~$\%$&99.64~$\%$&99.80~$\%$&99.84~$\%$&99.88~$\%$\\
				
				\rowcolor{Gray}
				Pegase 1354 &0.04~$\%$ &34.04~$\%$ &23.94~$\%$ &89.02~$\%$ &97.58~$\%$ &99.61~$\%$ &99.66~$\%$\\
				
				Polish 2383 &0.03~$\%$ &29.49~$\%$ &20.60~$\%$ &84.53~$\%$ &90.62~$\%$ &92.60~$\%$ &93.95~$\%$\\
				\hline
		\end{tabular}}
				\begin{tablenotes}\footnotesize
					\item IEEE test cases from~\cite{ABB-TX-KMG-KSS-TJO:17} and Pegase test case from~\cite{CJ-SF-JM-PP:16}.
				\end{tablenotes}
	\caption{Comparison of the conservativeness of various sufficient conditions with approximate synchronizations
			tests applied to IEEE test cases in the domain $S^{G}(\pi/2)$.}\label{tab:IEEE-test-cases}
	\end{table}
\end{center}


\paragraph{Summary evaluation} Table~\ref{tab:IEEE-test-cases} shows that, for IEEE test cases with scaled nominal power injections, the following statements holds:
\begin{enumerate}
\item the  accuracy of the approximate tests~\eqref{test:approx} increases with
  the order of the tests; 
\item for $\gamma=\frac{\pi}{2}$, the
  fifth and seventh order approximate tests~\eqref{test:approx} improves the accuracy given by the 1st order
  approximate test $ \norm{B^{\top}L^{\dagger}\omega}{\infty}\leq 1$
  \cite{FD-MC-FB:11v-pnas} by up to $9$\%;
\item for $\gamma=\frac{\pi}{2}$, the
  fifth and seventh order approximate tests~\eqref{test:approx}
  improves the best-known sufficient synchronization test in the
  literature~\cite{SJ-FB:16h} by up to $50$\%.
\end{enumerate}

Next, we consider random graph models with randomly generated natural
frequencies.  We setup the numerical analysis to assess the correctness of
the family of approximate tests~\eqref{test:approx} as follows.  Consider
a nominal unweighted random networks $ \{G,\omega\} $, where $G$ is a
connected undirected graph with $ n=80 $ nodes chosen from a parametrized
family of random graph models $\text{RGM}$ and $
\omega\in\vectorones[n]^{\perp} $ are natural frequencies chosen randomly
form sampling distribution $\text{SD}$. Then we study the synchronization
of the Kuramoto model with uniform coupling gain $ K\in\real_{>0} $,
\begin{equation*} \label{eq:kuramoto_K}
\dot{\theta} = \omega - KB\sin(B^{\top}\theta).
\end{equation*} 
The random graph models $\text{RGM}$ and the sampling distributions $\text{SD}$
are given as follows: 
\begin{enumerate}[nolistsep]
	\item
          \emph{Network topology:} For the network topology, we
          consider three types of random graph models $ \text{RGM}
          $. The random graph models we consider are: (i) \ErdosRenyi
          random graph with probability $ p $ of an edge
          existing~\cite{GC-XW-XL:15}, (ii) Random Geometric graph model with
          sampling region $(0,1]^2\subset\real^2$ and connectivity
          radius $ p $~\cite{GC-XW-XL:15}, and (iii) Watts\textendash{}Strogatz small world
          model network with initial coupling to the $ 2 $ nearest neighbors and rewiring probability $ p $ of an edge
          existing~\cite{DJW-SHS:98}. If there exists an edge, then the coupling
          weight is $a_{ij}=a_{ji}=1$. If the graph is not connected, then it is thrown out and a new random graph is generated.
	\item \emph{Natural frequencies:} We consider two types of
          sampling distributions SD. $ n=80 $ random numbers are
          sampled from either a (i) uniform distribution on the interval $ (-1,1) $ or (ii) bipolar distribution $ \{-1,+1\} $
          to obtain $ q_i $ for $ i\in\{1,\dots,n\} $. Then to ensure
          that t he natural frequencies satisfy $ \omega\in\vect{1}_n^{\perp} $, we take $ \omega_i = q_i - \sum_{i=1}^{n}q_i/n $.
	\item \emph{Parametric realizations:} We consider combinations
          of parameters $ (RGM,p,SD) $: the three random graph models,
          15 edge connectivity parameters on the interval  $ p\in(0,1) $, and two sampling distributions. 
\end{enumerate}
For each parametric realization in (iii), we generate $ 100 $ nominal models of $ \{ G,\omega\in\vect{1}_n^{\perp} \} $. For each nominal case, we find the critical coupling, denoted by $ \subscr{K}{C} $, and the smallest coupling where the approximate test fails \eqref{def:approxTest}, denoted by $ \subscr{K}{T} $, for orders $ \{1,3,5,7\} $. $ \subscr{K}{C} $ is found iteratively with MATLAB's \emph{fsolve}. We define the normalized critical coupling ratio of each random case by $ \subscr{K}{C}/\subscr{K}{T} $. The numerically determined values are found with an accuracy of $10^{-3} $. 
Each data point in Figure \ref{fig:series_approxSyncTest} corresponds to the mean of $ \subscr{K}{C}/\subscr{K}{T} $ over $ 100 $ nominal cases of the same parametric realization.

\begin{figure}[!htb]
	\begin{table}[H] \centering
		\begin{threeparttable}
			\begin{tabular}{c|c|c|} 
				&$ \omega $ uniform & $ \omega $ bipolar \\ \hline
				\rotatebox[origin=c]{90}{\ErdosRenyi Graph}
				& \begin{minipage}{.42\textwidth} \includegraphics[width=\textwidth]{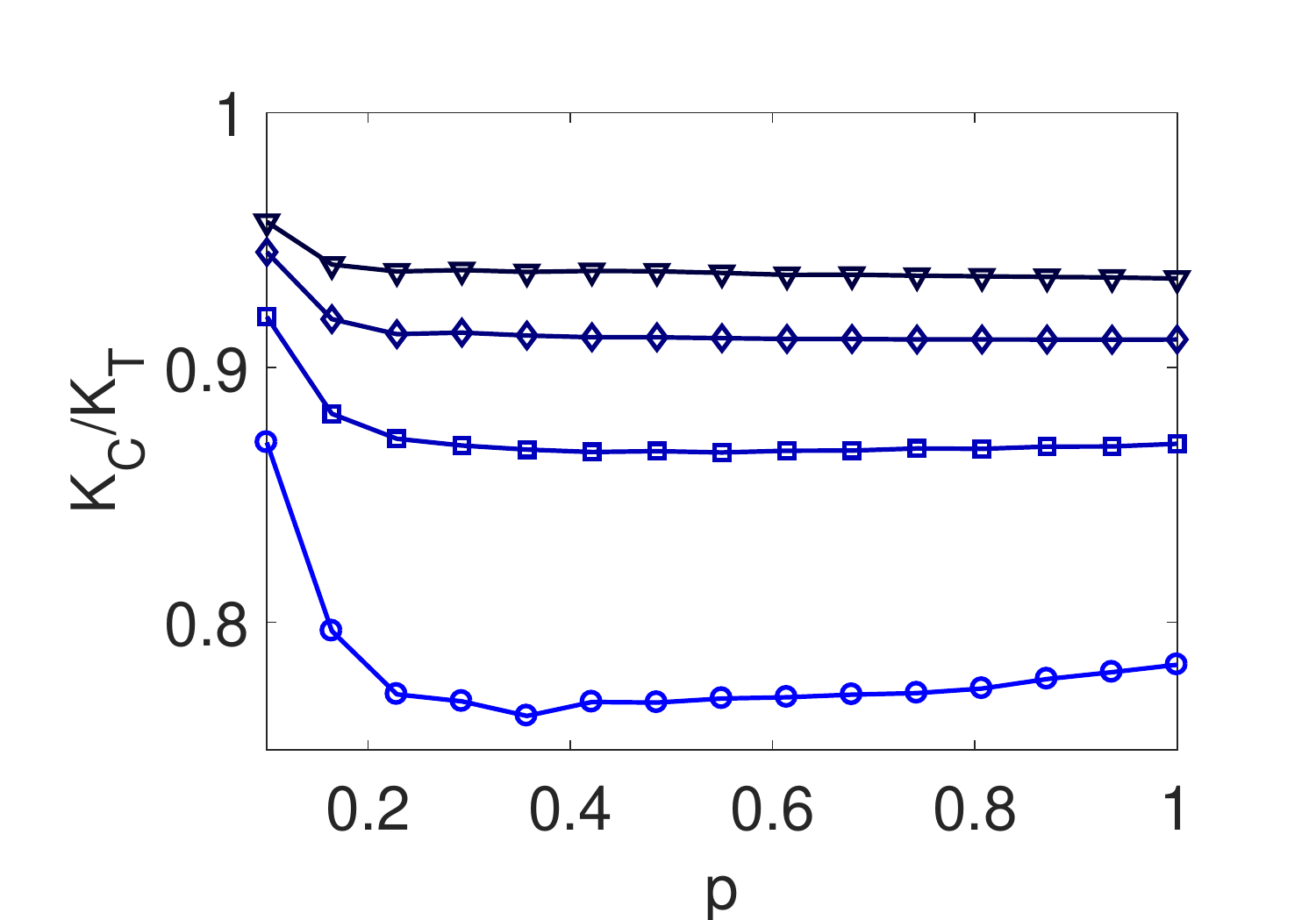} \end{minipage}
				& \begin{minipage}{.42\textwidth} \includegraphics[width=\textwidth]{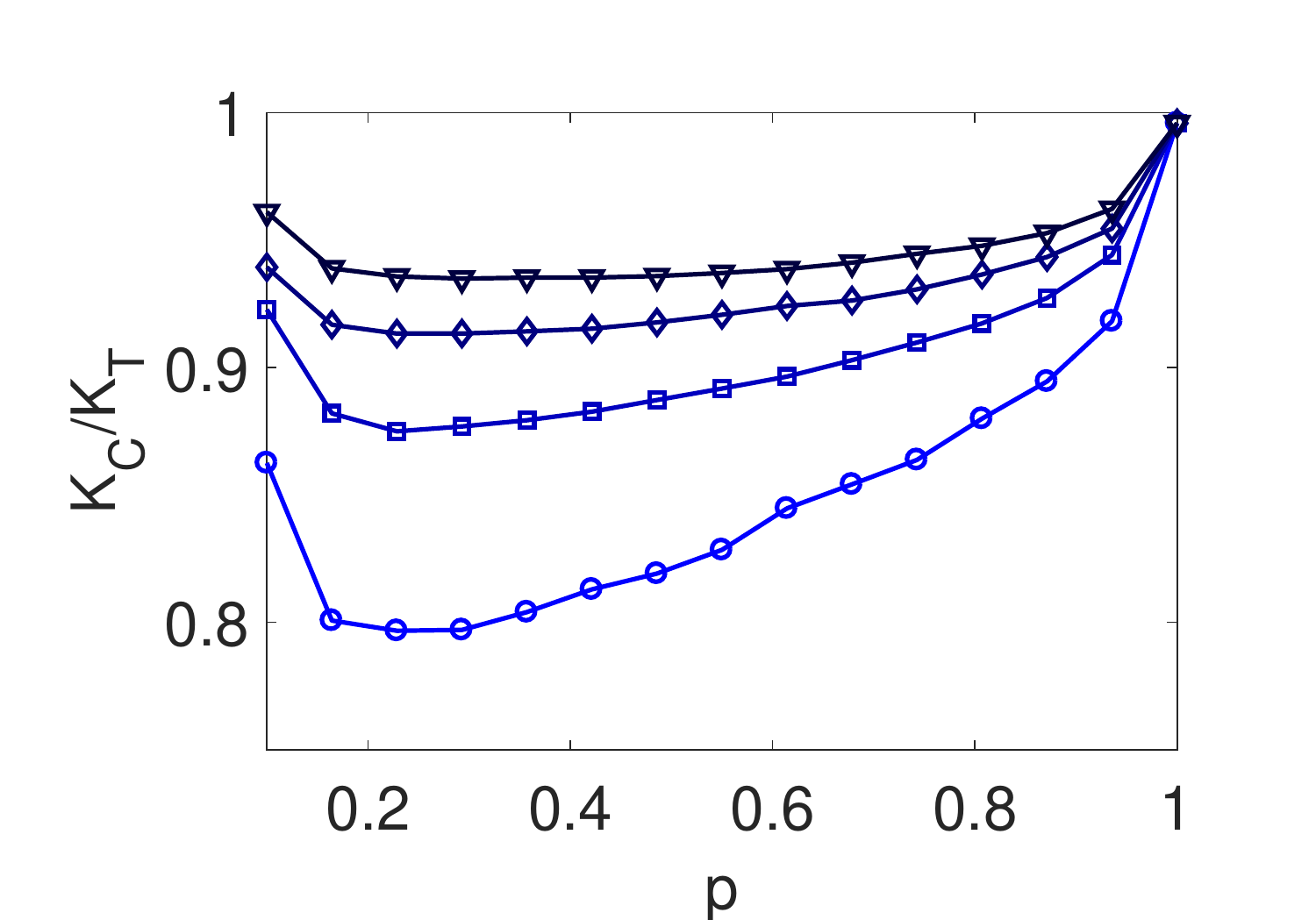} \end{minipage} \\ \hline
				\rotatebox[origin=c]{90}{Rnd. Geom. Graph}
				& \begin{minipage}{.42\textwidth}\includegraphics[width=\textwidth]{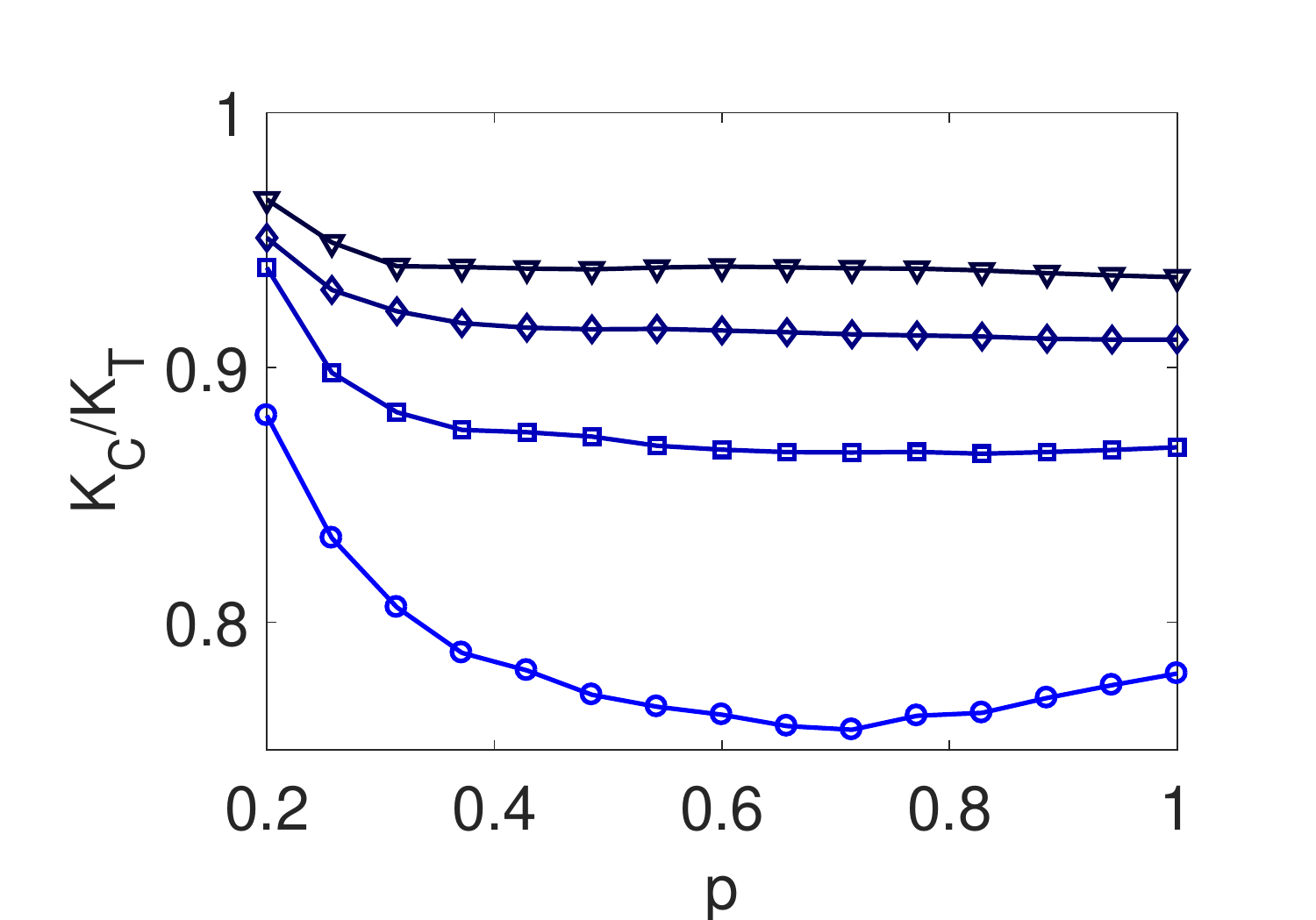} \end{minipage}
				& \begin{minipage}{.42\textwidth}\includegraphics[width=\textwidth]{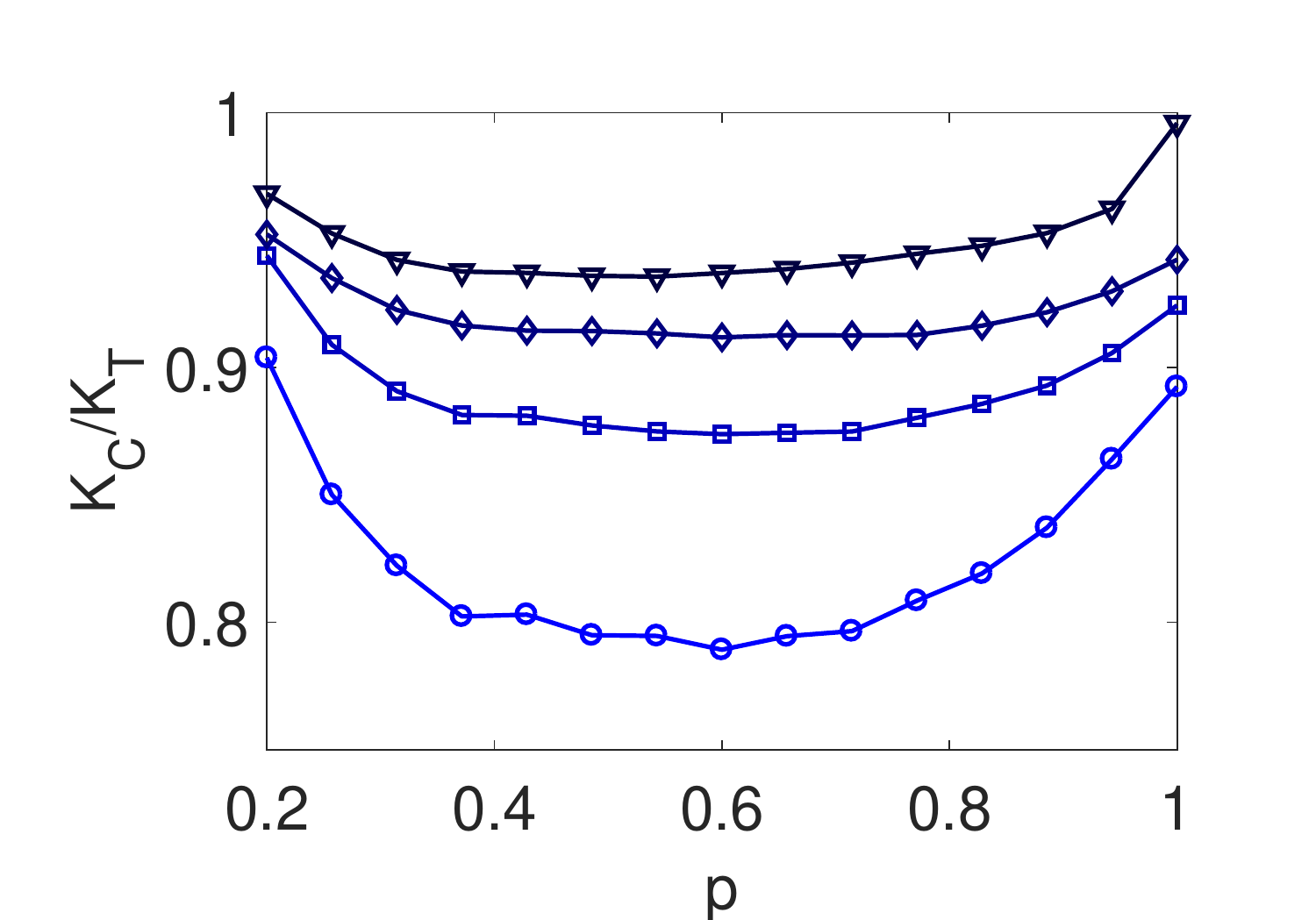} \end{minipage} \\ \hline
				\rotatebox[origin=c]{90}{Small World Ntwk.} 
				& \begin{minipage}{.42\textwidth}\includegraphics[width=\textwidth]{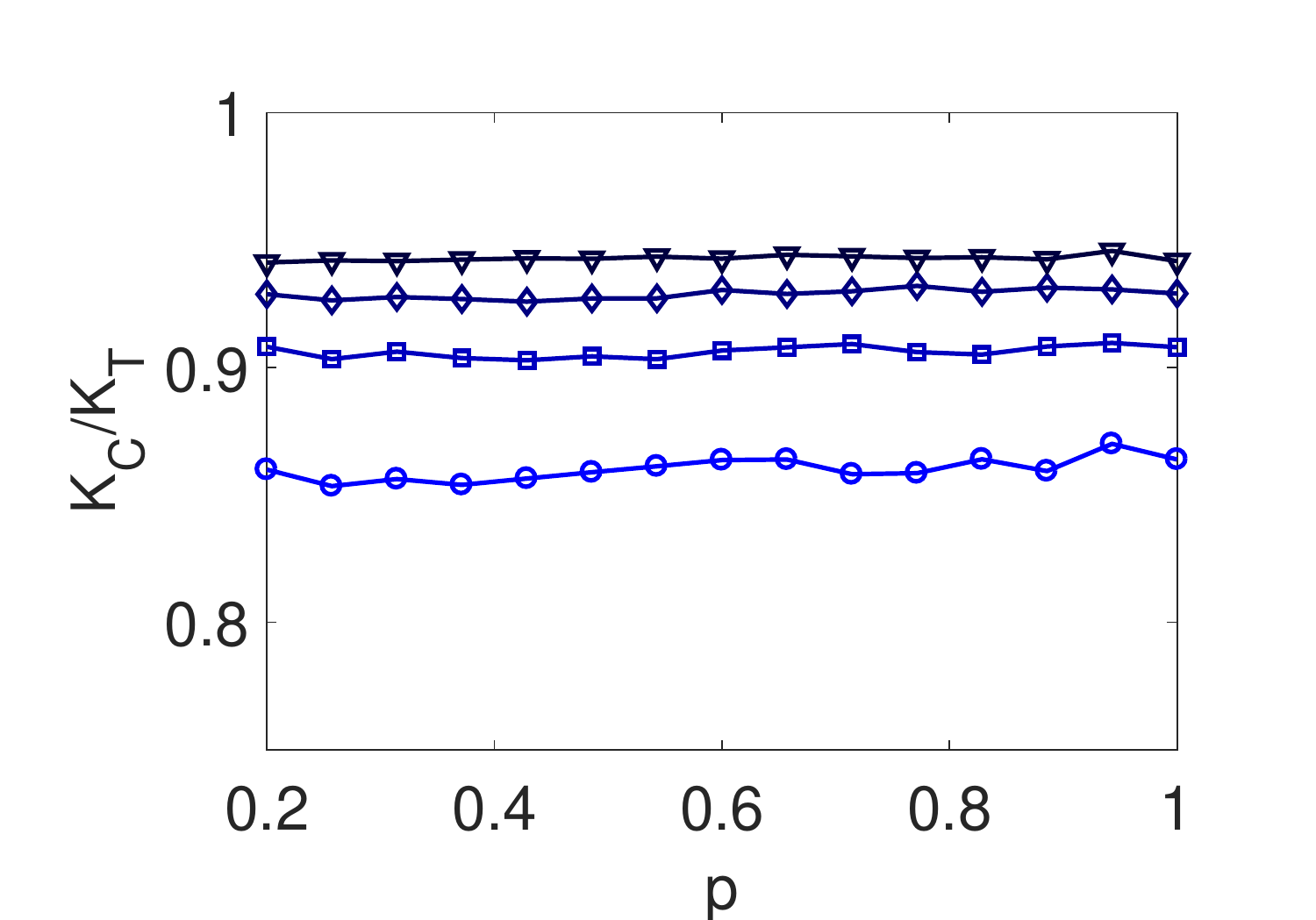} \end{minipage} 
				& \begin{minipage}{.42\textwidth}\includegraphics[width=\textwidth]{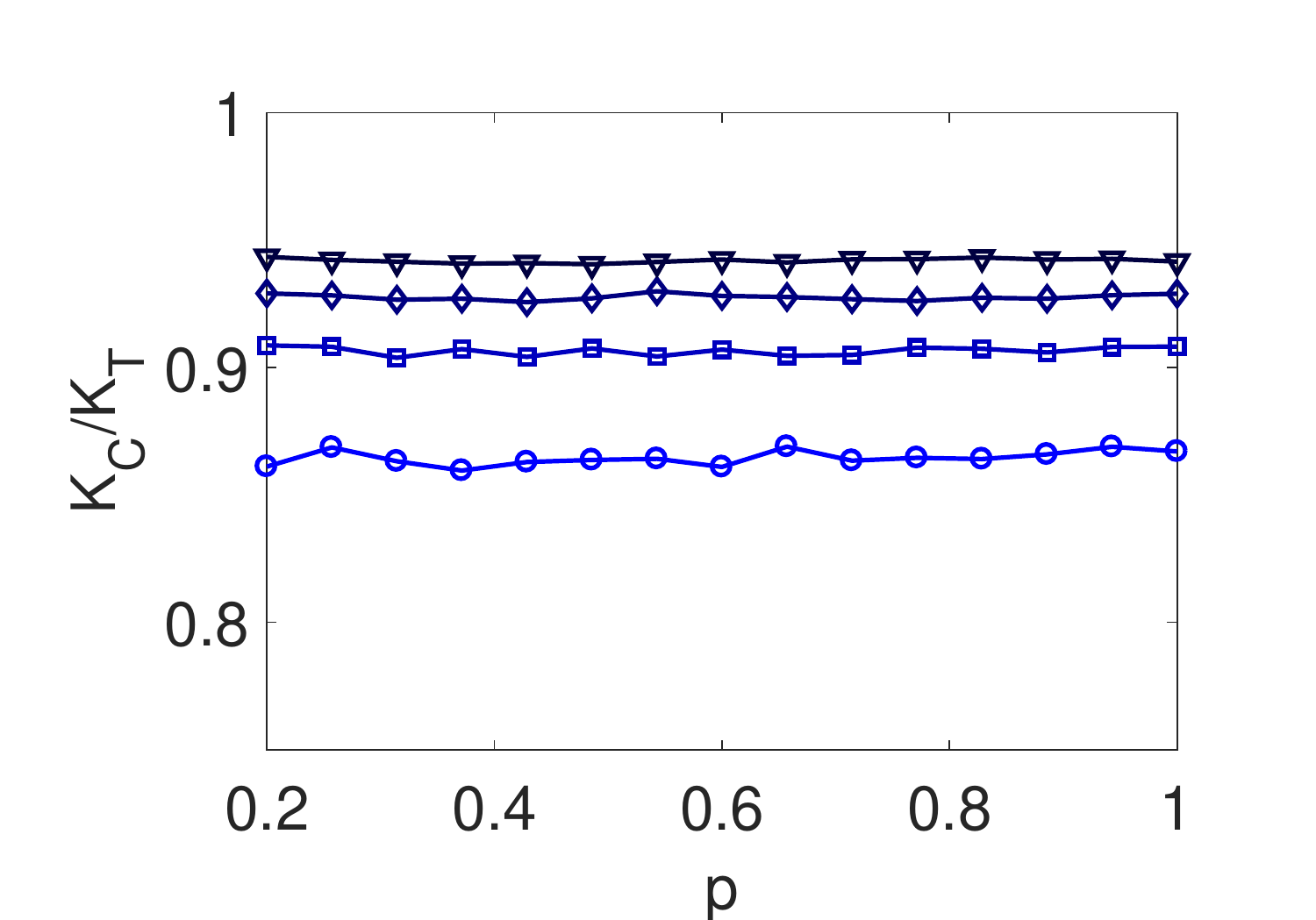} \end{minipage} \\\hline
			\end{tabular}
		\end{threeparttable}
	\end{table} 
	\vspace{-4mm}\hfil\begin{subfigure}{\textwidth} \centering
		\includegraphics[width=0.7\linewidth]{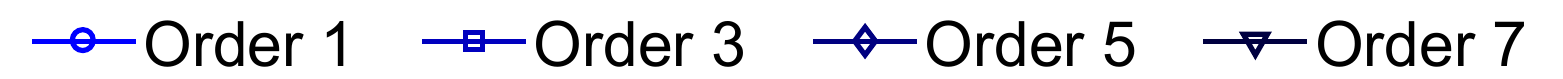}
	\end{subfigure} \vspace{-4mm}\hfil
	\caption{Each data point is the critical ratio $ \subscr{K}{C}/\subscr{K}{T} $ averaged over $ 100 $ random graphs with $n=80$ nodes and in the domain $S^{G}(\pi/2)$. $ \subscr{K}{C}/\subscr{K}{T} $ measures the accuracy of the approximate synchronization tests~\eqref{test:approx}.
		$\subscr{K}{C}$ is the smallest coupling gain such that there exists a solution to the Kuramoto model. $ \subscr{K}{T} $ is an approximation of $\subscr{K}{C}$, estimated using the approximate test~\eqref{test:approx} derived from Theorem~\ref{thm:suff_condition+power_series}\ref{p3:series} for orders $k={1,3,5,7}$.}
	\label{fig:series_approxSyncTest}
\end{figure}

\paragraph{Summary evaluation}
Figure~\eqref{fig:series_approxSyncTest} illustrates that, for random graph models with random natural frequency from
  bipolar and uniform distribution and $\gamma=\frac{\pi}{2}$, the
  accuracy of the approximate test~\eqref{test:approx} consistently improves. In particular, the fifth and seventh order approximate
  tests~\eqref{test:approx} improves the accuracy of the 1st order
  approximate test $\norm{B^{\top}L^{\dagger}\omega}{\infty}\leq 1$~\cite{FD-MC-FB:11v-pnas} by up to $30$\%.


\subsection{Computational cost of approximating the synchronization manifold}

Consider a connected graph $G$ with $m$ edges, $n$ nodes, and no self-loops.
Table~\ref{tab:compCost} shows the order of the number of operations
associated with different methods for approximating the
synchronization manifold for the sparse and dense graphs. 

\begin{table}[!htb] \centering
	\begin{threeparttable}
		\begin{tabular}{|c|c|c|c|} \hline
			Method 
			& General
			& Sparse Graphs
			& Dense Graphs \\
			
			&
			& $ \mathcal{O}(n)=\mathcal{O}(m) $
			& $ \mathcal{O}(n^2)=\mathcal{O}(m) $ \\ \hline
			\rowcolor{Gray}Precomputation 
			& $ \mathcal{O}(m^2n) $
			& $ \mathcal{O}(n^3) $
			& $ \mathcal{O}(n^5) $ \\ 
			Series, $ 5 $th order*
			& $ \mathcal{O}(2m^2) $
			& $ \mathcal{O}(2n^2) $
			& $ \mathcal{O}(2n^4) $ \\ 
			\rowcolor{Gray}Series, $ 7 $th order*
			& $ \mathcal{O}( 3m^2) $
			& $ \mathcal{O}(3n^2) $
			& $ \mathcal{O}(3n^4) $ \\ 
			Newton\textendash{}Raphson
			& $ \mathcal{O}(mn^2) $
			& $ \mathcal{O}(n^3) $
			& $ \mathcal{O}(n^4) $ \\ \hline
		\end{tabular}
		\begin{tablenotes}
			\small
			\item * Denotes that the method precomputes the terms $\Bt L^{\dagger}$, $ L^{\dagger}B\mcA $ and $\prjcyc$. The computation complexity of these terms are found in the ``Precomputation'' row.
			\item The computational complexity of $ L^\dagger $ for $ L\in\real^{n\times n} $ is $ \mathcal{O}(n^3) $.
		\end{tablenotes}
	\end{threeparttable}
\caption{Comparison of number of operations required for computing the truncated series and Newton\textendash{}Raphson.}\label{tab:compCost}
\end{table}

For random graph models, we compare the computational time of three different methods for approximating
the synchronization manifold of the Kuramoto model: (i) the series approximation of the
analytical solution from
Theorem~\ref{thm:suff_condition+power_series}\ref{p3:series}, (ii)
Newton\textendash{}Raphson method, and (iii) MATLAB's \emph{fsolve}.

For the simulation setup, we consider the random network
$\{G,\omega\}$ with $ n\geq2 $ nodes,  $ \omega_i \in (-\alpha,\alpha)
$ for $ i\in\{1,\dots,n\} $, and number of edges $m$ depending on the
coupling parameter $ p\in(0,1)$. The following lists the random graph parameters:
\begin{enumerate}[nolistsep]
\item\emph{Network topology:} To construct the random graph, the \ErdosRenyi random graph model was used with probability $p$ of an edge existing.
	If the graph is not connected, then it is thrown out and a new random graph is generated.
	\item \emph{Coupling weights:} Each edge is given a random coupling weight, $ a_{ij}=a_{ji}>0 $, sampled on the uniform distribution interval $ (0,10) $.
	\item \emph{Natural frequencies:} $ n $ random numbers are
          sampled from a uniform distribution on the interval $
          (-\alpha,\alpha) $ to obtain $ q_i $ for $ i\in\{1,\dots,n\}
          $. Then to ensure that the natural frequencies satisfy $
          \omega\in\vect{1}_n^{\perp} $, we take $ \omega_i = q_i -
          \sum_{i=1}^{n}q_i/n $. $ \alpha $ is chosen to be
          sufficiently small so that the MATLAB \emph{fsolve} converges to a
          solution of the Kuramoto model~\eqref{eq:kuramoto_model}.
	\item \emph{Parametric realizations:} We consider random network
          parametrization $(n,p,\alpha)$ with combinations of $
          n=\{10,20,30,60,120\} $ and $ p=\{0.2,0.4,0.6,0.8\} $.
\end{enumerate}
For each parametrization, we generate $ 3000 $ nominal graphs and $20$ natural frequency vectors for \emph{each} random graph.
The results of the execution time for various methods are shown in Figure \ref{fig:compTime_multw} where each point is the computational time for a particular method averaged over 3000 graphs and $ 20 $ natural frequency vectors $ \omega\in\vectorones[n]^{\perp} $ per graph.
The computation time for the series approximation is the total time to complete the calculations for $ 1 $ random nominal graph with $ 20 $ different natural frequency vectors. This time \emph{does not} include the computation time for the precomputed terms listed in Table \ref{tab:compCost}. The initial guess for Newton\textendash{}Raphson and \emph{fsolve} is $ B^{\top}L^{\dagger}\omega $. 

\begin{figure}[!htb]  
	\begin{subfigure}{\textwidth} \centering
		\includegraphics[width=\textwidth]{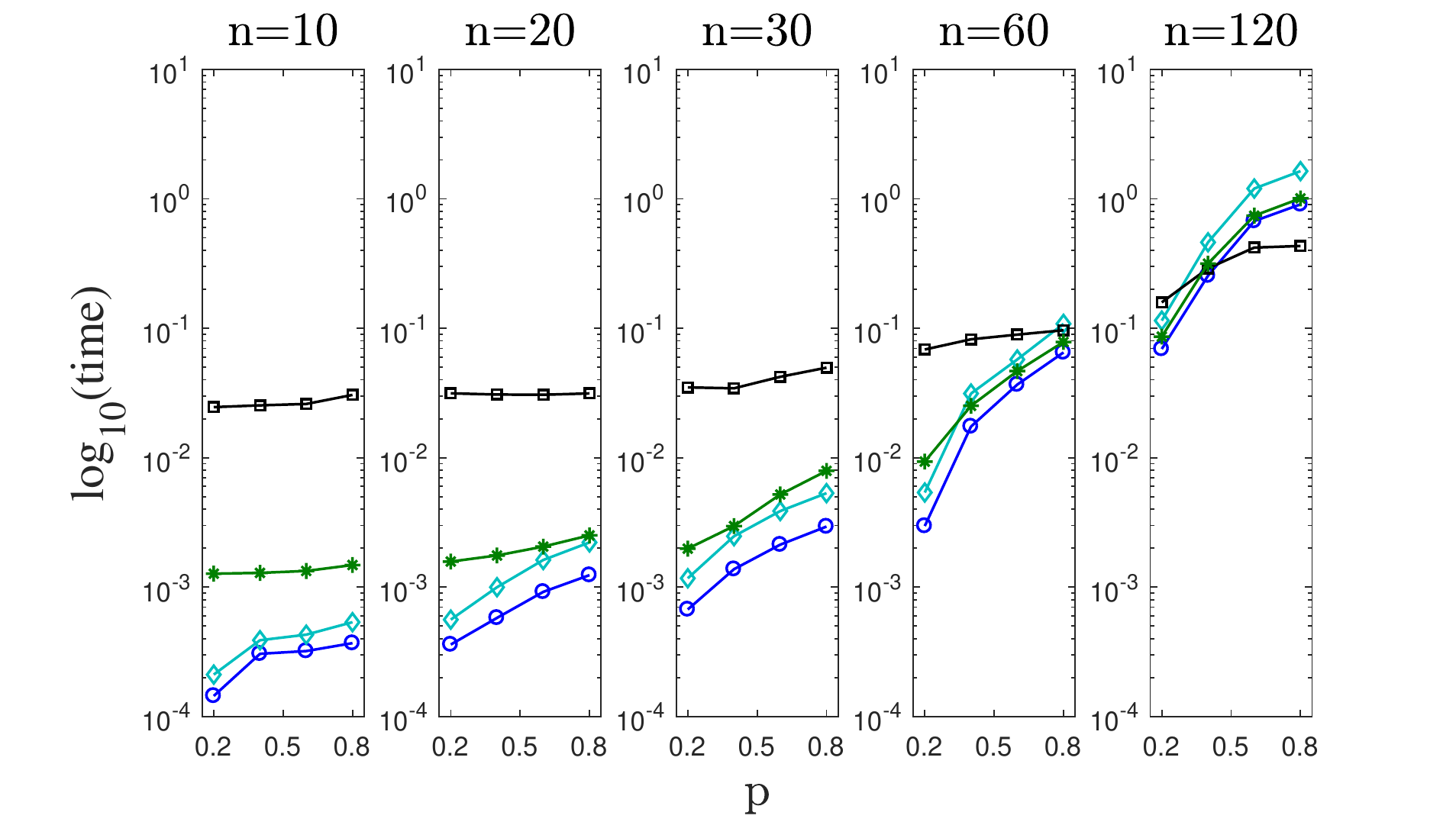}
	\end{subfigure} \begin{subfigure}{\textwidth} 
		\includegraphics[width=\textwidth]{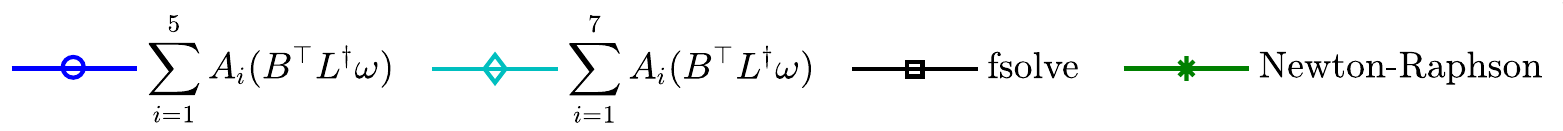}
	\end{subfigure}
\caption{Comparison of computation times for $5$th order truncated
  series, the $7$th order truncated series, MATLAB's \emph{fsolve},
  and Newton\textendash{}Raphson for random graphs, where $p$ is the probability
  of an edge existing for \ErdosRenyi graphs. The computation time is
  how long it takes the various methods to compute the solutions of
  the unconstrained edge balance equations, for $20$ randomly generated natural frequency vector given one randomly generated graph. If a random natural frequency vector does not give a solution, it is thrown out and a new vector is generated. Certain values are precomputed for each graph, but the precomputation time is not included in the graph. Each data point is averaged over $3000$ \ErdosRenyi random graphs. }\label{fig:compTime_multw} 
\end{figure}

\paragraph{Summary evaluation} Figure~\ref{fig:compTime_multw} show that the computation time for
the truncated power series increases with density of the random
graphs. Moreover, the truncated series are more efficient
than Newton\textendash{}Raphson method for small random graphs, while they
are only comparable to Newton\textendash{}Raphson method for large
random graphs. 

For IEEE test cases, we compare the computational time of three
different methods for
calculating the synchronization manifold: (i) the series approximation
of the analytical solution from
Theorem~\ref{thm:suff_condition+power_series}\ref{p3:series}, (ii)
Newton\textendash{}Raphson method, and (iii) MATLAB's \emph{fsolve}. The
setting for the IEEE test cases are the same as the one given in Section~\ref{sec:num-accuracy1}. In this setup we do not precompute any terms and consider one graph topology with its nominal power injections. We use each method to solve for the synchronization manifold $ [\theta^*] $, and average the computation time over $ 10 $ trials.

\begin{table}[!htb] \centering
		\begin{tabular}{|c|c|c|c|c|} \hline 
			Test Case
			& \emph{fsolve} / NR
			& Ord. 5 / NR 
			& Ord. 7 / NR \\ \hline
			\rowcolor{Gray}IEEE 118
			& 4.2072
			& 0.4511
			& 0.4539 \\ 
			IEEE 300
			& 2.6501
			& 0.7546
			& 0.7539 \\ 
			\rowcolor{Gray}Pegase 1354
			& 1.2825
			& 0.8582
			& 0.8633 \\ 
			Polish 2383
			& 1.1279
			& 0.9559
			& 0.9583 \\ \hline
		\end{tabular}
\caption{Computational times of MATLAB's \emph{fsolve}, fifth order series (Ord. 5), and 7th order series
  (Ord. 7) normalized by the computational time for Newton\textendash{}Raphson
  (NR) for IEEE test cases. The computation time is how long it takes
  the various methods to compute the solution of the unconstrained edge
  balance equations, averaged over 10 trials. Certain values are precomputed for each graph, but the precomputation times are not included.}\label{tab:compTime_ieee}
\end{table}

\paragraph{Summary evaluation} The results are found in Table
\ref{tab:compTime_ieee} show that, for IEEE test cases, the
series approximations are computationally comparable to Newton\textendash{}Raphson.

\section{Conclusion} 
This paper proposes a novel equivalent characterization of the equilibrium
equation for the Kuramoto coupled oscillator; we refer to this characterization
as to the unconstrained edge balance equation.  Using this characterization, we
propose a Taylor series expansion for the synchronization manifold of the
Kuramoto network and a recursive formula to symbolically compute all the terms
in the Taylor series. We then use the truncated Taylor series as a tool to (i)
find sharp approximation for the synchronization manifold and (ii) estimate the
onset of frequency synchronization. Our numerical simulations illustrate the
accuracy and computational efficiency of this method on various classes of
random graphs and IEEE test cases.  As future directions, it may be instructive
to employ this series expansion method to study frequency synchronization in
networks consisting of other important oscillators, such as
FitzHugh\textendash{}Nagumo systems. Additionally, it may be viable to adopt the
series expansion approach to tackle more general nonlinear network flow
problems, such as the coupled power flow equations and optimal power flow
problems.

\appendix

\section{Mathematica Code}\label{app:mathematics_code}

In this appendix, we present an implementation of a Mathematica algorithm
to compute the coefficient of the power series expansion given in
Theorem~\ref{thm:suff_condition+power_series}\ref{p3:series}.
 
\begin{algorithm} \caption{Mathematica algorithm to compute terms of series expansion}
	\label{algorithm_series}
	\begin{verbatim}
		permCount[as_]:=Length[Permutations[Flatten[as]]]
		getOddPartition[x1_,x2_]:=Select[IntegerPartitions[x1,{x2}],allOddQ]
		getSymbol[val_]:=Symbol["A"<>ToString[val]]
		A[1]=getSymbol[1]
		A[i_/;OddQ[i]]:=
		   -Pcyc**(Sum[(2k-1)!!/((2k+1)(2k)!!)*
		   Sum[permCount[as]*Product[getSymbol[a],{a,as}],
		   {as,getOddPartition[i,2k+1]}],{k,1,(i-1)/2}])
	\end{verbatim}
\end{algorithm}

It is worth mentioning that the required computations increase
exponentially with the order of the terms. Specifically, computing the
$(2k+1)$th order coefficient of the power series requires finding all the
odd-integer partitions of $2k+1$.


\bibliographystyle{plainurl}
\bibliography{alias,Main,FB}
\end{document}